\documentclass{article}

\usepackage[hidelinks]{hyperref}

\usepackage[margin=3.2cm]{geometry}

\usepackage{amssymb}
\usepackage{amsmath}
\usepackage{amsthm}
\usepackage{mathrsfs}

\newtheorem{theorem}{Theorem}[section]
\newtheorem{proposition}[theorem]{Proposition}
\newtheorem{lemma}[theorem]{Lemma}
\newtheorem{corollary}[theorem]{Corollary}

\theoremstyle{definition}
\newtheorem{definition}[theorem]{Definition}

\usepackage{xargs}
\usepackage{ifthen}

\include{thref}
\usepackage{accents}

\newcommand{\textdef}[1]{{\normalfont\textit{#1}}}
\newcommand{\st}{{}\,:\,{}}
\newcommand{\defiff}{:\hspace{-1mm}\iff}

\makeatletter
\def\function{\@ifstar\@function\@@function}
\newcommand{\@function}[2]{#1 \to #2}
\newcommand{\@@function}[2]{\colon #1 \to #2}

\def\mfunction{\@ifstar\@mfunction\@@mfunction}
\newcommand{\@mfunction}[2]{#1 \rightrightarrows #2}
\newcommand{\@@mfunction}[2]{\colon #1 \rightrightarrows #2}
\makeatother

\newcommand{\pfunction}[2]{:\subseteq #1 \to #2}
\newcommand{\pmfunction}[2]{:\subseteq #1 \rightrightarrows #2}

\newcommand{\charfun}[1]{\chi_{#1}}

\newcommand{\dom}{\operatorname{dom}}
\newcommand{\ran}{\operatorname{ran}}

\newcommand{\restrict}[1]{\ensuremath{\left. \hspace{-1mm} \right|_{#1}}}

\newcommand{\coding}[1]{\langle #1 \rangle}
\newcommand{\str}[1]{({#1})}
\newcommand{\length}[1]{|{#1}|}
\newcommand{\concat}{\smash{\raisebox{.9ex}{\ensuremath\smallfrown} }}

\newcommand{\finStrings}[1]{{#1}^{<\mathbb{N}}}
\newcommand{\infStrings}[1]{{#1}^{\mathbb{N}}}

\newcommand{\Baire}{{\mathbb{N}^\mathbb{N}}}
\newcommand{\baire}{{\finStrings{\mathbb{N}}}}

\newcommand{\Cantor}{{2^\mathbb{N}}}
\newcommand{\cantor}{{\finStrings{2}}}

\newcommand{\sequence}[2]{(#1)_{#2}}
\newcommand{\family}[2]{\{#1\}_{#2}}

\newcommand{\boldfaceDelta}{\underaccent{\sim}{\boldsymbol{\Delta}}}
\newcommand{\boldfaceSigma}{\underaccent{\sim}{\boldsymbol{\Sigma}}}
\newcommand{\boldfacePi}{\underaccent{\sim}{\boldsymbol{\Pi}}}
\newcommand{\boldfaceGamma}{\underaccent{\sim}{\boldsymbol{\Gamma}}}

\newcommand{\Borel}{\underaccent{\sim}{\mathbf{B}}}

\newcommand{\lightfaceDelta}{\Delta}
\newcommand{\lightfaceSigma}{\Sigma}
\newcommand{\lightfacePi}{\Pi}
\newcommand{\lightfaceGamma}{\Gamma}

\newcommand{\hmetric}{\operatorname{d}_\mathcal{H}}

\newcommand{\hypClosed}{\boldfacePi^0_1}
\newcommand{\hypClosedF}{\mathbf{F}}
\newcommand{\hypCompact}{\mathcal{K}}

\newcommand{\hypCompactV}{\mathbf{K}}
\newcommand{\hypCompactUF}{\mathbf{K}_{U}}

\newcommand{\hypClosedUF}{\mathbf{F}_U}

\newcommand{\hypClosedUV}{\mathbf{F}_{UV}}

\newcommand{\closedNegRep}{\psi_-}
\newcommand{\closedPosRep}{\psi_+}
\newcommand{\closedFullRep}{\psi}

\newcommand{\compactCover}{\kappa_-}
\newcommand{\compactMinCover}{\kappa}

\newcommand{\cauchyRep}{\delta_{\mathcal{H}}}

\newcommand{\wadgereducible}{\le_W}

\newcommand{\setdifference}{\setminus}
\newcommand{\setcomplement}[2][]{\ifthenelse{\equal{#1}{}}{#2^\mathrm{C}}{ {#1}\setdifference {#2} }}

\newcommand{\norm}[2]{{ \left\|\, #1\, \right\|_{#2} }}

\newcommand{\diam}[1]{\mathrm{diam}(#1)}

\newcommand{\ball}[2]{{B\left(#1,#2\right)}}

\newcommand{\cball}[2]{\overline{\ball{#1}{#2}}}

\newcommand{\closure}[1]{\overline{#1}}

\newcommand{\boundary}{\partial}

\newcommand{\proj}{\operatorname{proj}}

\newcommand{\support}{\operatorname{spt}}

\newcommand{\dd}[1]{\mathop{d#1}}

\newcommand{\hmeas}{\mathcal{H}}
\newcommand{\hdim}{ \operatorname{dim}_\mathcal{H} }

\newcommand{\fourierdim}{\operatorname{dim}_{\mathrm{F}}}

\newcommand{\fouriertransform}[1]{{ \widehat{#1} }}
\newcommand{\scalarprod}[2]{{ {#1}\cdot {#2} }}

\newcommand{\Salem}{\mathscr{S}}

\newcommand{\ProbabilityMeas}{\mathbb{P}}
\newcommand{\BorelMeas}

\newcommand{\repmap}[1]{\delta_{#1}}
\newcommand{\reptop}[1]{\mathscr{O}(#1)}

\newcommand{\Sier}{\mathbb{S}}

\newcommand{\leftReal}{\mathbb{R}_<}
\newcommand{\rightReal}{\mathbb{R}_>}

\newcommand{\mflim}{\mathsf{lim}}

\newcommand{\weireducible}{\le_{\mathrm{W}}}

\newcommand{\weiequiv}{\equiv_{\mathrm{W}}}

\newcommand{\compproduct}{*}

\newcommandx{\cont}[3][1={},2={}]{{ \mathcal{C}^{#1}_{#2}\ifthenelse{\equal{#3}{}}{}{(#3)} }}
\newcommand{\torus}{\mathbb{T}}

\usepackage{color}

\usepackage{url}
\usepackage{doi}

\author{Alberto Marcone \and Manlio Valenti}
\date{}

\newcommand{\printauthor}{{
\bigskip
\footnotesize

Alberto Marcone, \textsc{Department of Mathematics, Computer Science and Physics\newline
University of Udine\newline
Udine, UD 33100, IT}\par\nopagebreak
\textit{E-mail address}: \url{alberto.marcone@uniud.it}

\medskip

Manlio Valenti, \textsc{Department of Mathematics, Computer Science and Physics\newline
University of Udine\newline
Udine, UD 33100, IT}\par\nopagebreak
\textit{E-mail address}: \url{manliovalenti@gmail.com}
}}

\def\subjclass[#1]#2{
	\begingroup
	\renewcommand\thefootnote{}\footnote{\textup{#1} \textit{Mathematics Subject Classification}: #2}%
	\addtocounter{footnote}{-1}%
	\endgroup
}

\title{Effective aspects of Hausdorff and Fourier dimension}

\begin{document}

%

\maketitle

\begin{abstract}
In this paper, we study Hausdorff and Fourier dimension from the point of view of effective descriptive set theory and Type-2 Theory of Effectivity. Working in the hyperspace $\hypCompactV(X)$ of compact subsets of $X$, with $X=[0,1]^d$ or $X=\mathbb{R}^d$, we characterize the complexity of the
family of sets having sufficiently large Hausdorff or Fourier dimension. This, in turn, allows us to show that the family of all the closed Salem sets is $\lightfacePi^0_3$-complete. One of our main tools is a careful analysis of the effectiveness of a classical theorem of Kaufman. We furthermore
compute the Weihrauch degree of the functions computing Hausdorff and Fourier dimension of closed sets.
\end{abstract}

\subjclass[2020]{Primary: 03D78; Secondary: 03D55, 28A75, 28A78 }

\tableofcontents

\section{Introduction}

Hausdorff dimension is probably the most important and well studied among the notions of fractal dimensions, and it plays a central role in analysis and geometric measure theory. In recent work, Jack and Neil Lutz \cite{LutzLutz2017} proved a point-to-set principle linking the (classical)
Hausdorff dimension of a set with the (relative) effective Hausdorff dimension of its points. If we restrict our attention to singletons, we can characterize the effective Hausdorff dimension of $\{ \xi \}$ by means of the Kolmogorov complexity of $\xi$ \cite{LutzMayordomo}, which establishes a
surprising connection between two (apparently) very distant notions.
 
A powerful tool to study the Hausdorff dimension of Borel subsets of $\mathbb{R}^d$ is provided by the Fourier transform. Indeed, Frostman's lemma draws an interesting connection between the Hausdorff dimension of a set and the decay of the Fourier transform of a (probability) measure supported on it. This leads to the notion of Fourier dimension. It is known that the Fourier dimension of a Borel set cannot exceed its Hausdorff dimension.

This work is part of a long-term effort, involving many researchers, aimed at exploring the recursion-theoretic properties of the Fourier dimension. While no point-to-set principles can hold for the Fourier dimension (in such a generality), analyzing the complexity of the Fourier dimension in simpler cases can shed light on the general behavior of the Fourier dimension itself (up to now, still not deeply understood).

Salem sets arise naturally when combining geometric measure theory and harmonic analysis. A set $A\subset \mathbb{R}^d$ is called \textdef{Salem} iff $\hdim(A)=\fourierdim(A)$, where $\hdim$ and $\fourierdim$ denote the Hausdorff and the Fourier dimension respectively. Explicit (i.e.\ non-random)
Salem sets are not easy to build. A classic example comes from the theory of Diophantine approximation of real numbers: for every $\alpha \ge 0$, the set $E(\alpha)$ of $\alpha$-well approximable numbers is Salem with dimension $2/(2+\alpha)$. The computation of its Hausdorff dimension is due to Jarn\'ik
\cite{Jarnik1928} and Besicovitch \cite{Besicovitch1934}, while the result on its Fourier dimension is due to Kaufman \cite{Kaufman81}. The reader is referred to \cite{Bluhm98} or \cite{Wolff03} for detailed proofs of Kaufman's theorem. The construction presented in \cite{Wolff03} will play a central role in the rest of this work.

As a consequence of a result of Gatesoupe \cite{Gatesoupe67}, subsets of $\mathbb{R}^d$ obtained by the rotation of a $1$-dimensional Salem sets with dimension $\alpha$ (having at least two points) are Salem of dimension $d-1+\alpha$, and this provides a simple way for building Salem sets of dimension at least $d-1$ in $\mathbb{R}^d$.

Explicit examples of non-Salem sets are the symmetric Cantor sets with dissection ratio $1/n$ for $n>2$: they are known to have null Fourier dimension and Hausdorff dimension $\log(2)/\log(n)$ (see \cite[Sec.\
4.10]{Mattila95} and \cite[Thm.\ 8.1]{MattilaFA}). Every subset of a $n$-dimensional hyperplane is a $0$-Fourier dimensional subset of $\mathbb{R}^d$ when $n<d$, while it can have any Hausdorff dimension up to $n$.

In recent work \cite{MVSalem}, we studied the complexity, from the point of view of classical descriptive set theory, of a number of relations involving the Hausdorff and the Fourier dimension. In particular, we studied the conditions $\hdim(A)> p$, $\fourierdim(A)>p$, $\hdim(A)\ge p$, $\fourierdim(A)\ge p$, ``$A$ is Salem", when $p\in \mathbb{R}$ and $A$ is a closed subset of $[0,1]$, $[0,1]^d$, and $\mathbb{R}^d$. In particular, we proved that having Hausdorff/Fourier dimension $>p$ and $\ge p$ are, respectively, a $\boldfaceSigma^0_2$-complete and a $\boldfacePi^0_3$-complete conditions. Similarly, we showed that the family of Salem sets is $\boldfacePi^0_3$-complete. In this paper we explore the same conditions from the point of view of effective descriptive set theory and Type-2 Theory of Effectivity (TTE). Notice that in \cite{HLT2007} the authors showed that the sets of elements of the Cantor space having (respectively) effective Hausdorff dimension $>\alpha$ and $\ge \alpha$, where $\alpha$ is a $\Delta^0_2$-computable real, are (respectively) $\lightfaceSigma^0_2$ and $\lightfacePi^0_3$.\smallskip

The paper is organized as follows. After briefly introducing the relevant background notions
(Section~\ref{sec:background}), we present some results on the lightface structure of the hyperspaces of closed and compact sets (Section~\ref{sec:hyp_closed_compact}) and on computable measure theory (Section~\ref{sec:comp_measure_theory}) that will be needed for the main results. In particular, we show that the hyperspace of compact subsets of a computably compact space is computably compact (\thref{thm:K(X)_computably_compact}), the hyperspace of closed subsets of the Euclidean space is computably compact (\thref{thm:hypclosed_Rd_eff_compact}) and that the space of probability measures on a computably compact space is computably compact (\thref{thm:computability_measures}). Section~\ref{sec:eff_kaufman} is devoted to the proof of an effective version of the above mentioned theorem by Kaufman, stated in \thref{thm:s_alpha_effective}. Section~\ref{sec:eff_complexity_salem} contains the main results on the effective complexity of the conditions mentioned above, and can be summarized  as follows: if $X=[0,1]^d$ or $X=\mathbb{R}^d$ then,
\begin{center}
	\renewcommand{\arraystretch}{1.5}
	\begin{tabular}{ | c | c | c | }
		\hline
		$p<d$ & $\{ A \in\hypCompactV(X) \st \hdim(A)> p\}$ & $\lightfaceSigma^0_2$-complete \\
		$p>d$ & $\{ A \in\hypCompactV(X) \st \hdim(A)\ge p\}$ & $\lightfacePi^0_3$-complete \\
		$p<d$ & $\{ A \in\hypCompactV(X) \st \fourierdim(A)> p\}$& $\lightfaceSigma^0_2$-complete \\
		$p>d$ & $\{ A \in\hypCompactV(X) \st \fourierdim(A)\ge p\}$& $\lightfacePi^0_3$-complete \\
		\hline
		\multicolumn{2}{| c |}{ $\{ A \in\hypCompactV(X)\st A \text{ is Salem}\}$ }&$\lightfacePi^0_3$-complete \\
		\hline
	\end{tabular}
\end{center}
where $\hypCompactV(X)$ is the hyperspace of compact subsets of $X$ (endowed with the canonical lightface structure induced by the Hausdorff metric, introduced in Section~\ref{sec:hyp_closed_compact}). The complexities remain the same if we consider the hyperspace of closed sets. In particular, the
fact that the family of closed Salem subsets of $[0,1]$ is $\lightfacePi^0_3$-complete answers a question asked by Slaman during the IMS Graduate Summer School in Logic, held in Singapore in 2018. In
Section~\ref{sec:salem_weihrauch}, we use our results to characterize the Weihrauch degree of the maps computing the Hausdorff and the Fourier dimension of a closed set, in particular answering a question raised by Fouch\'e (\cite{Dagstuhl2016}) and Pauly.

\subsubsection*{Acknowledgements}
The early investigations leading to this paper were conducted jointly with Ted Slaman and Jan Reimann. We would also like to thank Vasco Brattka, Antonio Montalb\'an, Arno Pauly, Matthias Schr\"oder, Luciano Tubaro, and Linda Brown Westrick for useful discussions and suggestions on the topics of the paper. We thank the two anonymous referees for their careful reading of the paper.

Both author's research was partially supported by the Italian PRIN 2017 Grant ``Mathematical Logic: models, sets, computability".

\section{Background}
\label{sec:background}
Throughout the paper, we will use $\ball{x}{r}$ to denote the open ball with center $x$ and radius $r$. We also fix a computable enumeration $\sequence{q_i}{i\in\mathbb{N}}$ of $\mathbb{Q}^+$.
\subsection{Hausdorff and Fourier dimension}
\label{sec:background_hausdorr_fourier}
Let us briefly introduce the relevant notions from geometric measure theory. For a more thorough presentation the reader is referred to \cite{Falconer14}.

Let $(X,d)$ be a separable metric space and let $A\subset X$. We denote the diameter of $A$ by $\diam{A}$. For every $s\ge 0$, $\delta\in (0,+\infty]$ we define
\begin{align*}
	\hmeas^s_\delta(A) & := \inf\left\{ \sum_{i\in I} \diam{E_i}^s \st \family{E_i}{i\in I} \text{ is a }\delta\text{-cover of }A \right\},\\
	\hmeas^s(A) & := \lim_{\delta \to 0^+} \hmeas^s_\delta(A) = \sup_{\delta>0} \hmeas^s_\delta(A),
\end{align*}
where $\family{E_i}{i\in I}$ is a $\delta$-cover of $A$ if $A\subset \bigcup_{i\in I} E_i$ and $\diam{E_i}\le \delta$ for each $i\in I$. The function $\hmeas^s$ is called \textdef{$s$-dimensional Hausdorff measure}. The \textdef{Hausdorff dimension} of $A$ is defined as
\[ 	\hdim(A):= \sup\{ s\in [0,+\infty)\st \hmeas^s(A) >0\}. \]

It is well-known that, as a consequence of Frostman's lemma (see \cite[Thm.\ 8.8]{Mattila95}), the Hausdorff dimension of a Borel subset of $\mathbb{R}^d$ can be equivalently written as
\[ \sup \{ s\in [0,d] \st (\exists \mu \in \ProbabilityMeas(A))(\exists c>0)(\forall x \in \mathbb{R}^d)(\forall r>0)(\mu(\ball{x}{r})\le cr^s)\}, \]
where $\ProbabilityMeas(A)$ is the set of Borel probability measures
supported on $A$ (in other words, for Borel sets Hausdorff and capacitary
dimensions coincide).

This characterization suggests the possibility to use the tools of harmonic analysis to obtain estimates on the Hausdorff dimension. We can define the \textdef{Fourier transform of a probability measure $\mu\in \ProbabilityMeas(\mathbb{R}^d)$} as the function
\[ \fouriertransform{\mu}\colon\mathbb{R}^d\to \mathbb{C}:= x \mapsto \int_{\mathbb{R}^d} e^{-i\, \scalarprod{\xi}{x}} \dd{\mu(x)} \]
where $\scalarprod{\xi}{x}$ denotes scalar product. The \textdef{Fourier
dimension} of $A\subset \mathbb{R}^d$ is then defined as
\[ \fourierdim(A):= \sup\{ s\in [0,d]\st (\exists \mu\in \ProbabilityMeas(A))(\exists c>0)(\forall x\in \mathbb{R}^d)(|\fouriertransform{\mu}(x)|\le c|x|^{-s/2} ) \}. \]

It is known that, for every Borel $A\subset\mathbb{R}^d$, $\fourierdim(A)\le \hdim(A)$ (see \cite[Chap.\ 12]{Mattila95}). If $\fourierdim(A)= \hdim(A)$ then $A$ is called \textdef{Salem set}. We denote the collection of Salem subsets of $X\subset \mathbb{R}^d$ with $\Salem(X)$.

For background notions on the Fourier transform the reader is referred to \cite{SteinWeiss}. For its applications to geometric measure theory see \cite{MattilaFA}.

We notice that the Hausdorff dimension is
\begin{description}
	\item[countably stable]: for every family $\family{A_i}{i\in\mathbb{N}}$, $\hdim(\bigcup_i A_i) = \sup_i \hdim(A_i)$ \cite[p.\ 59]{Mattila95};
	\item[invariant under bi-Lipschitz maps]: for every $\alpha$-H\"older continuous map $f\colon\mathbb{R}^n\to \mathbb{R}^m$ we have $\hdim(f(A))\le \alpha^{-1}\hdim(A)$ \cite[Prop.\ 3.3]{Falconer14}.
\end{description}
In particular, the inclusion map $\iota\function{\mathbb{R}^n}{\mathbb{R}^m}$, with $n\le m$, preserves the Hausdorff dimension.

None of the above properties hold, in full generality, for the Fourier
dimension. In fact, it is not even finitely stable \cite[Sec.\ 1.3]{EPS2014}
and does not behave well under H\"older continuous transformations
\cite[Sec.\ 8]{ESSurveyFourier}. Moreover, it is sensitive to the choice of
the ambient space: as mentioned, every $A\subset \mathbb{R}^n$ has null
Fourier dimension when seen as a subset of $\mathbb{R}^m$ with $n<m$.
However, some regularity properties hold if we restrict our attention to
special cases. We mention, in particular, that the Fourier dimension is
\textdef{inner regular for compact sets}, i.e.
\[ \fourierdim(A) = \sup \{ \fourierdim(K) \st K \subset A \text{ and } K \text{ is compact} \}, \]
and countably stable for closed sets \cite[Prop.\ 5]{EPS2014}, i.e.\ for every countable family $\family{A_k}{k}$ of closed subsets of $\mathbb{R}^d$ we have
\[ \fourierdim\left(\bigcup_{k} A_k\right) = \sup_k \fourierdim(A_k). \]
Moreover, the Fourier dimension is invariant under similarities or affine (invertible) transformations (this is a simple consequence of the properties of the Fourier transform).

\subsection{Computability on represented spaces}
In this paper, we use the standard approach of Type-2 Theory of Effectivity (TTE) to define a notion of computability on a wide range of spaces. We now introduce the main definitions, for a more detailed presentation the reader is referred to \cite{Pauly16,Weihrauch00}.

Let $\Baire$ be the Baire space and let $\baire$ be the set of finite sequences of natural numbers. Let also $\Cantor$ be the Cantor space and $\cantor$ be the set of finite binary sequences. Both $\Baire$ and $\Cantor$ are endowed with the usual product topology. We sometimes describe a string by a list of its elements. E.g.\ we write $\str{n_0, n_1,\hdots, n_k }$ for the string $\sigma:= i \mapsto n_i$. Similarly, we can describe an infinite string by $\str{n_0, n_1,\hdots}$, when it is clear from the context how to continue the sequence. We write $\str{}$ for the empty sequence. We write $\length{\sigma}$ for the length of $\sigma$ and $\sigma\concat \tau$ for the concatenation of the strings $\sigma$ and $\tau$. We will use the symbol $\coding{\cdot}$ to denote a fixed computable bijection $\function*{\baire}{\mathbb{N}}$ with computable inverse. It is often convenient to write $\coding{n_0,\hdots, n_k}$ in place of $\coding{\str{n_0,\hdots, n_k}}$. In the literature, the symbol $\coding{\cdot}$ is often used to denote also the \textdef{join} between two (of the same length, finite or infinite) strings. With a (relatively) small abuse of notation, if $x$,$y$ are two strings of the same length we will write\footnotemark{} $\coding{x,y}(i):= \coding{x(i),y(i)}$ and $\coding{x_0,x_1,\hdots}(\coding{i,j}) := x_i(j)$.
\footnotetext{The exact details of the definition of the join are often not relevant. E.g.\ a common way to define the join of $x,y\in \Baire$ is letting $\coding{x,y}(2n):= x(n)$ and $\coding{x,y}(2n+1):= y(n)$.}%

A \textdef{represented space} is a pair $(X,\repmap{X})$ where $X$ is a set and $\repmap{X}\pfunction{\Baire}{X}$ is a partial surjection called \textdef{representation map}. For every $x\in X$, the elements of $\repmap{X}^{-1}(x)$ are called $\repmap{X}$-\textdef{names} for $x$ (we just say names if there is no ambiguity on the representation map).

We can exploit the classical notion of computability on $\Baire$ (see
\cite{Weihrauch00} for an introduction) to induce a notion of computability
on any represented space: let $f$ be a partial multi-valued function between
the represented spaces $(X,\repmap{X})$ and $(Y,\repmap{Y})$ (in symbols
$f\pmfunction{X}{Y}$). A \textdef{realizer} for $f$ is a partial function
$F\pfunction{\Baire}{\Baire}$ such that, for every $p\in\dom(f\circ
\repmap{X})$, we have that $\repmap{Y}(F(p))\in f(\repmap{X}(p))$. We say
that $f$ is \textdef{$(\repmap{X},\repmap{Y})$-computable} if it has a
computable realizer (again, we just say computable if the representation maps
are clear from the context). We say that $f$ is \textdef{realizer-continuous}
if it has a continuous realizer. The set of realizer-continuous partial functions between represented spaces is a represented space itself (\cite[Sec.\ 2.3]{Weihrauch00}).

As the notation suggests, the induced notion of computability is intrinsically tied to the choice of the representation maps. If $\delta$ and $\delta'$ are two representation maps for $X$, we say that $\delta$ is \textdef{(topologically) reducible} to $\delta'$, and we write $\delta\le \delta'$ (resp.\ $\delta\le_t \delta'$), if there is a (continuous) computable map $F\pfunction{\Baire}{\Baire}$ s.t.\ $\delta(p)=\delta'(F(p))$ for every $p\in \dom(\delta)$. The maps $\delta$ and $\delta'$ are called \textdef{(topologically) equivalent}, written $\delta\equiv \delta'$ (resp.\ $\delta \equiv_t \delta'$), if $\delta\le\delta'$ and $\delta'\le\delta$ (resp.\ $\delta\le_t\delta'$ and $\delta'\le_t\delta$).

Often times, spaces are naturally endowed with some canonical topology, and it would be desirable that the topological structure agrees with the computational one, i.e.\ that the notions of continuity and realizer-continuity agree. We will consider (and mainly focus our attention on) the so-called \textdef{admissible representations}, which intuitively are those that satisfy this requirement.

\begin{definition}[{\cite[Def.\ 1]{Schroder02}}]
	\thlabel{def:admissibility}
	Let $(X,\tau_X)$ be a topological space. A representation map $\repmap{X}$ of $X$ is called \textdef{admissible w.r.t.} $\tau_X$ if it is continuous and, for every other continuous representation map $\delta$ on $X$, we have $\delta \le_t \repmap{X}$.
\end{definition}%

In other words, an admissible representation of $X$ is $\le_t$-maximal among the continuous representation of $X$. We will just say that a representation is admissible if there is no ambiguity on the topology.

\begin{theorem}[{\cite[Thm.\ 3.2.11]{Weihrauch00}}]
	Let $(X,\repmap{X},\tau_X)$, $(Y,\repmap{Y},\tau_Y)$ be admissibly represented second-countable $T_0$ spaces. For every $f\pfunction{X}{Y}$,
	\[ f \text{ is continuous} \iff f \text{ is realizer-continuous.}  \]
\end{theorem}

In particular, whenever $X$ and $Y$ are admissibly represented, the space $\cont[][]{X,Y}$ of continuous functions from $X$ to $Y$ is a represented space.

Clearly, the very existence of an admissible representation for the
topological space $(X,\tau_X)$ depends on the topology itself: a family
$\mathcal{B}$ of subsets of $X$ is called a \textdef{pseudobase} iff for
every open set $U\subset X$, every $x\in U$ and every sequence
$(y_n)_{n\in\mathbb{N}}$ converging to $x$,
\[ (\exists B\in\mathcal{B})(\exists n_0 \in \mathbb{N})( \{x\}\cup \{y_n\st n\ge n_0\} \subset B\subset U). \]

\begin{theorem}[{\cite[Thm.\ 13]{Schroder02}}]
	\thlabel{thm:characterization_admissibility}
	A topological space $(X,\tau_X)$ admits an admissible representation $\repmap{X}$ iff it is $T_0$ and admits a countable pseudobase.
\end{theorem}

\subsection{Representations on (hyper)spaces}
\label{sec:rep_hyperspaces}

While (the proof of) \thref{thm:characterization_admissibility} provides an explicit definition of an admissible representation for a wide variety of spaces, in many practical situations this is not the representation map we endow our space with. 
Observe that, while all admissible representation maps are topologically equivalent, they do not necessarily induce the same notions of computability, i.e.\ they may not be (computably) equivalent.

Let $X=(X,d,\alpha)$ be a separable metric space, where $d\function{X\times X}{\mathbb{R}}$ is the distance function and $\alpha\function{\mathbb{N}}{X}$ is an enumeration of a dense subset of $X$. The \textdef{Cauchy representation} on $X$ is the map $\repmap{\mathsf{C}}\pfunction{\Baire}{X}$
defined as	 
\[ \repmap{\mathsf{C}}(p) = x \defiff \lim_{n\to\infty} \alpha(p(n)) = x,  \]
where $\dom(\repmap{\mathsf{C}}):= \{ p\in\Baire \st  (\forall n)(\forall m>n)(|\alpha(p(n)) - \alpha(p(m))| \le 2^{-n}) \}$ is the set of \textdef{rapidly converging sequences}. The Cauchy representation is the ``canonical'' representation map for separable metric spaces. It is equivalent to the representation map that names $x\in X$ via any $q\in\Baire$ s.t.\ $\{x\}= \bigcap_{n\in\mathbb{N}} \ball{\alpha(q(n))}{2^{-n}}$.

Recall that $\sequence{q_i}{i\in\mathbb{N}}$ is a canonical computable enumeration of the rationals. We say that $X$ is a \textdef{computable metric space} if the set
\[ \{ (i,j,n,m)\in \mathbb{N}^4\st q_i < d(\alpha(n),\alpha(m)) < q_j \} \]
is computably enumerable, i.e.\ if the restriction of the distance function to $\ran(\alpha)^2$ is computable. We can always assume that, if $X$ is infinite, $\alpha$ is an injective map (i.e.\ every element of the dense subset of $X$ has a unique index). Indeed, for every infinite computable
metric space $(X,d,\alpha)$ there is an injective subsequence $\beta$ of $\alpha$ s.t.\ the spaces $(X,d,\alpha)$ and $(X,d,\beta)$ are computably homeomorphic (i.e.\ there is a computable bijection with computable inverse) \cite[Thm.\ 2.9]{GregKispPauly17}.

Let $(Y,\tau_Y)$ be a second-countable topological space. We say that $(Y,\sequence{B_n}{n\in\mathbb{N}})$ is an \textdef{effective (topological) space} if $\sequence{B_n}{n\in\mathbb{N}}$ is an enumeration of a basis for $\tau_Y$ s.t.\ there is a computable function $\varphi\function{\baire\times \mathbb{N}}{\mathbb{N}}$ s.t.
\[ \bigcap_{i<\length{\sigma}} B_{\sigma(i)} = \bigcup_{k\in\mathbb{N}} B_{\varphi(\sigma,k)}. \]
Every effective topological space $(Y,\sequence{B_n}{n\in\mathbb{N}})$ can be endowed with the structure of represented space by defining an (admissible) representation map $\repmap{Y}$ that names a point $y\in Y$ via an enumeration of the set $\{i \st y\in B_i\}$. Notice that a computable metric space $(X,d,\alpha)$ can be seen as an effective space by considering the ``standard'' enumeration of the basis for $X$ (i.e.\ $B_{\coding{i,j}}=\ball{\alpha(i)}{q_j}$). In this case, the Cauchy representation on $X$ is equivalent to the representation $\repmap{X}$.

For every topological space $Z$, the family of Borel subsets of $Z$ can be stratified in a hierarchy, called the \textdef{Borel hierarchy}. The levels of this hierarchy are defined by transfinite recursion on $1\le \xi< \omega_1$, where $\omega_1$ is the first uncountable ordinal. We denote by $\boldfaceSigma^0_1(Z)$ and $\boldfacePi^0_1(Z)$ respectively the family of the open and the closed subsets of $Z$. For every $\xi>1$ we define:
\begin{itemize}
	\item[] $\boldfaceSigma^0_\xi(Z) := \left\{\bigcup_n A_n \setdifference B_n \st A_n,B_n\in \boldfaceSigma^0_{\xi_n}(Z),\, \xi_n< \xi,\, n\in\mathbb{N}\right\} $,
	\item[] $\boldfacePi^0_\xi(Z) := \{ X\setdifference A \st A \in \boldfaceSigma^0_\xi(Z) \} $.
\end{itemize}
Moreover, for every $\xi$, we define $\boldfaceDelta^0_\xi(Z) := \boldfaceSigma^0_\xi(Z) \cap \boldfacePi^0_\xi(Z)$. In particular, $\boldfaceDelta^0_1(Z)$ is the family of clopen subsets of $Y$. The families $\boldfaceSigma^0_2(Z)$ and $\boldfacePi^0_2(Z)$ are often written resp.\ $\boldsymbol{F}_\sigma(Z)$ and $\boldsymbol{G}_\delta(Z)$. It is known that $\Borel(Z)=\bigcup_\xi \boldfaceSigma^0_\xi(Z) = \bigcup_\xi \boldfacePi^0_\xi(Z) = \bigcup_\xi \boldfaceDelta^0_\xi(Z)$, where $\Borel(Z)$ denotes the family of Borel subsets of $Z$. We will omit the dependency from the space if there is no ambiguity.

For metric spaces (in fact, for Hausdorff spaces), the definition of the pointclass $\boldfaceSigma^0_\xi$ can be simplified letting
\[ \boldfaceSigma^0_\xi = \left\{\bigcup_n A_n \st A_n \in \boldfacePi^0_{\xi_n},\, \xi_n< \xi,\, n\in\mathbb{N}\right\}.\]

Let $W$ and $Z$ be topological spaces and let $A\subset W$, $B\subset Z$. We say that $A$ is \textdef{Wadge reducible} to $B$ if there is a continuous function $f\function{W}{Z}$ s.t.\
\[ x\in A \iff f(x)\in B~.\]

Let $\boldfaceGamma$ be a Borel class and assume that $Z$ is Polish. We say that $B$ is \textdef{$\boldfaceGamma$-hard} if $A\wadgereducible B$ for every $A\in \boldfaceGamma(\Baire)$. If $B$ is $\boldfaceGamma$-hard and $B \in \boldfaceGamma(Z)$ then we say that $B$ is \textdef{$\boldfaceGamma$-complete}.

\medskip

For every effective second-countable space $(Y,\sequence{B_n}{n\in\mathbb{N}})$, we say that $A\subset Y$ is \textdef{effectively open} if $A=\bigcup_{n\in\mathbb{N}} B_{\varphi(n)}$ for some computable function $\varphi\function{\mathbb{N}}{\mathbb{N}}$. The set of effectively open subsets of $Y$ is denoted by $\lightfaceSigma^0_1(Y)$.  In other words, an effective open set is a computable union of basic open sets. The complement of an effectively open set is called \textdef{effectively closed} and the family of all effectively closed subsets of $Y$ is denoted by $\lightfacePi^0_1(Y)$.

Notice that $\lightfaceSigma^0_1(Y)$ sets can be indexed using the code for a computable function defining them. In other words, there is a canonical indexing $\sequence{A_i}{i\in\mathbb{N}}$ of the $\lightfaceSigma^0_1(Y)$ sets. This allows us to define
\begin{itemize}
	\item[] $\lightfaceSigma^0_2(Y):= \{ A \subset Y \st A=\bigcup_{n\in\mathbb{N}} A_{\varphi(2n+1)}\setdifference A_{\varphi(2n)}, \text{ for some computable }\varphi \}$;
	\item[] $\lightfacePi^0_{2}(Y) := \{ Y \setdifference A \st A \in \lightfaceSigma^0_{2}(Y) \} $.
\end{itemize}
We can inductively define the \textdef{(Kleene's) arithmetical hierarchy}, also called \textdef{lightface hierarchy}, by letting $\sequence{A^n_i}{i\in\mathbb{N}}$ be an effective indexing of the $\lightfaceSigma^0_n(Y)$ sets and defining
\begin{itemize}
	\item[] $\lightfaceSigma^0_{n+1}(Y):= \{ A \subset Y \st A=\bigcup_{i\in\mathbb{N}} A^n_{\varphi(2n+1)}\setdifference A^n_{\varphi(2n)}, \text{ for some computable }\varphi \}$;
	\item[] $\lightfacePi^0_{n+1}(Y) := \{ Y\setdifference A \st A \in \lightfaceSigma^0_{n+1}(Y) \} $.
\end{itemize}
The lightface hierarchy can be relativized in a straightforward manner, by
defining, for $z \in \Cantor$,
\[ \lightfaceSigma^{0,z}_1(Y) := \left\{ A \subset Y \st A=\bigcup_{n\in\mathbb{N}} B_{f(n)} \text{ for some }z\text{-computable function }f \right\}, \]
and then, define the classes $\lightfacePi^{0,z}_n$, $\lightfaceSigma^{0,z}_{n+1}$, $\lightfaceDelta^{0,z}_n$ accordingly. It is important to mention that the lightface classes are universal for their corresponding boldface ones. Formally, if $\lightfaceGamma$ is a lightface class among $\lightfaceSigma^0_n,\lightfacePi^0_n$ and $\boldfaceGamma$ is the corresponding boldface pointclass, then
\[ P\in \boldfaceGamma(Y) \iff (\exists z\in\Baire)(P \in \lightfaceGamma^z(Y)), \]
see e.g.\ \cite[Thm.\ 3E.4]{MoschovakisDST}.

\smallskip

For every effective space $(Y,\sequence{B_n}{n\in\mathbb{N}})$ and every $k\ge 1$, we can define the represented spaces $(\boldfaceSigma^0_k,\repmap{\boldfaceSigma^0_k})$, $(\boldfacePi^0_k,\repmap{\boldfacePi^0_k})$, $(\boldfaceDelta^0_k,\repmap{\boldfaceDelta^0_k})$ inductively by:
\begin{itemize}
	\item $\repmap{\boldfaceSigma^0_1}(p):=\bigcup_{i\in\ran(p)} B_i$;
	\item $\repmap{\boldfacePi^0_k}(p):=\setcomplement[Y]{\repmap{\boldfaceSigma^0_k}(p)}$;
	\item $\repmap{\boldfaceSigma^0_{k+1}}(\coding{p_0,q_0,p_1,q_1,\hdots}):=\bigcup_{i\in\mathbb{N}} \repmap{\boldfaceSigma^0_k}(p_i)\setminus \repmap{\boldfaceSigma^0_k}(q_i)$;
	\item $\repmap{\boldfaceDelta^0_k}(\coding{p,q}):=\repmap{\boldfaceSigma^0_k}(p)$, iff $p,q\in\dom(\repmap{\boldfaceSigma^0_k})$ and $\repmap{\boldfaceSigma^0_k}(p)=\setcomplement[Y]{\repmap{\boldfaceSigma^0_k}(q)}$.
\end{itemize}

We notice that the $\lightfaceSigma^0_k$ sets are exactly those having a computable $\repmap{\boldfaceSigma^0_k}$-name (and similarly for $\lightfacePi^0_k$, $\lightfaceDelta^0_k$).

If $Y$ and $Y'$ are effective spaces, we say that $A\subset Y'$ is \textdef{effectively Wadge reducible} to $B\subset Y$, and write $A \le_m B$, if there is a recursive functional $f\function{Y'}{Y}$ s.t.\ $x\in A$ iff $f(x)\in B$. Let $\lightfaceGamma$ be a lightface pointclass as above and assume that $Y$ is an effective Polish space. We say that $B$ is \textdef{$\lightfaceGamma$-hard} if $A \le_m B$ for every $A\in \lightfaceGamma(\Baire)$. If $B$ is $\lightfaceGamma$-hard and $B \in \lightfaceGamma(Y)$ then we say that $B$ is \textdef{$\lightfaceGamma$-complete}. Standard examples\footnote{The standard proofs showing that the listed sets are $\boldfaceGamma$-complete for their respective class are, in fact, effective. See \cite[Sec.\ 23.A]{KechrisCDST}.} of $\lightfaceGamma$-complete sets are the following:
\begin{center}
	\renewcommand{\arraystretch}{1.5}
	\begin{tabular}{  l l  }
		$Q_2:= \{ x\in \Cantor \st (\forall^\infty m)(x(m)=0)\}$ & $\lightfaceSigma^0_2$-complete, \\
		$N_2:= \{ x\in \Cantor \st (\exists^\infty m)(x(m)=0)\}$ & $\lightfacePi^0_2$-complete, \\
		$S_3:= \{ x\in 2^{\mathbb{N}\times \mathbb{N}} \st (\exists k)(\exists^\infty m)(x(k,m)=0)\}$ & $\lightfaceSigma^0_3$-complete, \\
		$P_3:= \{ x\in 2^{\mathbb{N}\times \mathbb{N}} \st (\forall k)(\forall^\infty m)(x(k,m)=0)\}$ & $\lightfacePi^0_3$-complete,
	\end{tabular}
\end{center}
where $(\exists^\infty m)$ and $(\forall^\infty m)$ mean respectively $(\forall n)(\exists m\ge n)$ and $(\exists n)(\forall m\ge n)$.

While often there is a natural choice for an effective basis, when working with represented spaces we can exploit the representation map to induce a lightface structure in a canonical way.

Let us introduce the \textdef{Sierpi\'nski space} $\Sier := \{ 0,1 \}$. The
space $\Sier$ is endowed with the topology $\{ \emptyset, \{1\}, \Sier\}$.
This space is represented as follows: the only name for $0$ is the string
that is constantly $0$, while every other string in $\Baire$ is a name for
$1$.

We can notice that, if $(X,\repmap{X})$ is a represented space and $\reptop{X}$ is the final topology on $X$ induced by $\repmap{X}$, then the open sets $U\in\reptop{X}$ are exactly the subsets of $X$ s.t.\ the characteristic function $\charfun{U}\function{X}{\Sier}$ is realizer-continuous (see also \cite[Sec.\ 4]{Pauly16}). In particular, we can represent an open set $U\in \reptop{X}$ using a name for $\charfun{U}$. This, in turn, allows us to represent a closed set (in the final topology on $X$) via a name for its complement. Using the jumps of the Sierpi\'nski space, we can obtain an analogous characterization for the pointclasses $\boldfaceSigma^0_\xi(X)$ (\cite[Sec.\ III and Prop.\ 30]{PdBDST2015}, see also \cite{deBrecht2014}).

In other words, using the Sierpi\'nski space, we can define a representation map for the sets $\boldfaceSigma^0_k$, $\boldfacePi^0_k$, $\boldfaceDelta^0_k$, for any represented space $(X,\repmap{X})$. For separable metric spaces, the two representations are equivalent (see \cite{Pauly16,Brattka05}).

The same ideas allow us to induce a lightface structure on any represented space. Indeed, for a represented space $(X,\repmap{X})$, we can define the effectively open sets as follows:
\[ A\in \lightfaceSigma^0_1(X) \defiff \text{the characteristic function } \chi_A\function{X}{\Sier} \text{ of }A \text{ is computable.}\]
The Sierpi\'nski space is, thus, useful to obtain a notion of semi-decidability in represented spaces. For a more detailed discussion the reader is referred to \cite{CallardHoyrup2020,PdBDST2015, Pauly14}.

\section{The hyperspaces of closed and compact sets}
\label{sec:hyp_closed_compact}

The hyperspaces of closed and compact sets will play a crucial role in this
paper. The space $\boldfacePi^0_1(X)$ of closed subsets of a topological
space $X$ is usually endowed with a number of topologies we now recall. For a
more thorough presentation, the reader is referred to
\cite{Beer1993,KleinThom84}.

Let us define
\begin{gather*}
	\mathscr{U}:=\Big\{ \{F\in \hypClosed(X) \st F \cap C = \emptyset \} \st C\in \hypClosed(X) \Big\} ,\\
	\mathscr{L}:=\Big\{ \{F\in \hypClosed(X) \st F \cap U \neq \emptyset\} \st U\in \boldfaceSigma^0_1(X) \Big\}.
\end{gather*}
The topology $\tau_{UV}$ having $\mathscr{U}$ as a prebase is called \textdef{upper topology} or \textdef{upper Vietoris topology}, while the topology $\tau_{L}$ having $\mathscr{L}$ as a prebase is called \textdef{lower topology} or \textdef{lower Vietoris topology} (\cite[Def.\ 1.3.1 and def.\ 1.3.2]{KleinThom84}). The \textdef{Vietoris topology} $\tau_V$ is the topology having as a prebase the family $\mathscr{U} \cup \mathscr{L}$.

We also consider the collection $\mathscr{U}_{\hypCompact}$ defined as
\[ \mathscr{U}_{\hypCompact}:=\Big\{  \{F\in \hypClosed(X) \st F \cap K = \emptyset \} \st K\in \hypCompact(X) \Big\}, \]
where $\hypCompact(X)$ is the family of all compact subsets of $X$. The
family $\mathscr{U}_{\hypCompact}$ is a prebase for the topology $\tau_{UF}$
on $\hypClosed(X)$ called \textdef{upper Fell topology}. We define the
\textdef{Fell topology} $\tau_F$ on $\hypClosed(X)$ as the topology having as
a prebase the set $\mathscr{U}_{\hypCompact} \cup \mathscr{L}$. For this
reason, the lower Vietoris topology is also called \textdef{lower Fell
topology}.

If we restrict our attention to compact subsets of $X$, we can define the topological space $\hypCompactV(X)=(\hypCompact(X), \tau_V\restrict{\hypCompact(X)})$ obtained by endowing the family of compact subsets of $X$ with the topology induced by the Vietoris topology on $\hypClosed(X)$. This choice is motivated by the following observation: if $X$ is a metric space with distance $d$, we can define the \textdef{Hausdorff metric} $\hmetric$ on $\hypCompact(X)$ as follows:
\[ \hmetric(K,L) := \begin{cases}
	0 & \text{if } K=L=\emptyset\\
	1 & \text{if exactly one between }K \text{ and } L\text{ is }\emptyset\\
	\max\{\delta(K,L), \delta(L,K) \} &\text{otherwise }
\end{cases} \]
where $\delta(K,L) := \max_{x\in K} d(x,L)/(1+d(x,L))$. It is known that the Hausdorff metric $\hmetric$ is compatible with the Vietoris topology on $\mathcal{K}(X)$ (\cite[Ex.\ 4.21]{KechrisCDST}) and that if $X$ is Polish then so is $\hypCompactV(X)$ (\cite[Thm.\ 4.22]{KechrisCDST}).

We notice that $(\hypClosed(X),\tau_V)$ fails to be paracompact, and hence metrizable, if $X$ is not compact (\cite[Thm.\ 2]{Ke70}). The Fell topology is the preferred choice when working with closed sets, since if $X$ is Polish and locally compact then $(\hypClosed(X),\tau_F)$ is a Polish compact space and its Borel space is exactly the Effros-Borel space (\cite[Ex.\ 12.7]{KechrisCDST}). If $X$ is compact then the Fell and the Vietoris topologies coincide, and the same holds for the upper Fell and the upper Vietoris topologies.

\medskip

In the following, let $(X,d,\alpha)$ be a computable metric space. We already mentioned that the set $\boldfacePi^0_1(X)$ can be seen as a represented space, where a name for a closed set is a list of basic open balls that exhaust the complement. This representation map is often\footnote{see \cite[Def.\ 3.4]{BG2008}, \cite[Sec.\ 2]{BGP17}. In \cite{BrattkaPresser2003} it is introduced in Def.\ 3.5(1) and is denoted $\delta_{\mathrm{union}}$.} denoted $\closedNegRep$, and $\closedNegRep$-names provide \emph{negative information} on the set they represent. The represented space $(\hypClosed(X),\closedNegRep)$ is often denoted $\mathcal{A}(X)$ or $\mathcal{A}_-(X)$ in the literature.

In contrast, the \emph{positive information} representation $\closedPosRep$ for closed sets is defined as
\[ \closedPosRep(p)=F \defiff (\forall n\in \mathbb{N})(n+1 \in \ran(p) \iff F\cap B_n \neq \emptyset) ,\]
where $\sequence{B_n}{n\in\mathbb{N}}$ is the canonical enumeration of basic open balls of $X$. In other words, a $\closedPosRep$-name for $F$ is a list of all the basic open sets that intersect $F$ (\cite[Def.\ 3.1]{BG2008}, it is denoted $\delta^<$ in \cite[Def.\ 3.1(1)]{BrattkaPresser2003}). The elements of the represented space $(\hypClosed(X),\closedPosRep)$ are called \textdef{closed overt sets}, and the space is sometimes denoted $\mathcal{V}(X)$ in the literature, e.g.\ \cite[Sec.\ 2]{dBPSOvert}.

We denote the join of both positive and negative information representations with $\closedFullRep$. Formally
\[ \closedFullRep(\coding{p,q})=F \defiff \closedNegRep(p)=\closedPosRep(q)=F. \]
This is denoted $\psi_=$ in \cite[Def.\ 3.1]{BG2008} and $\delta^=$ in \cite[Def.\ 3.1]{BrattkaPresser2003}.

It is known that the representation maps $\closedNegRep$, $\closedPosRep$, and $\closedFullRep$ are admissible respectively for the upper Fell, lower Fell, and Fell topology on $\hypClosed(X)$ \cite[Sec.\ 3]{BG2008}.

We mention the following known facts.
\begin{theorem}[{\cite[Sec.\ 7]{BG2008}}]
	\thlabel{thm:computability_union_intersection}
	Let $X$ be a computable metric space. The following operations are computable:
	\begin{enumerate}
		\item $\cup\function{(\hypClosed(X),\delta)\times (\hypClosed(X),\delta)}{(\hypClosed(X),\delta)}:=(A,B)\mapsto A\cup B$, for $\delta \in \{ \closedNegRep, \closedPosRep, \closedFullRep \}$.
		\item $\bigcap\function{\infStrings{(\hypClosed(X),\closedNegRep)}}{(\hypClosed(X),\closedNegRep)}:=\sequence{A_n}{n\in\mathbb{N}}\mapsto \bigcap_{n\in\mathbb{N}} A_n$.
	\end{enumerate}
\end{theorem}

Notice however that $\cap\function{(\hypClosed(X),\closedFullRep)\times (\hypClosed(X),\closedFullRep)}{(\hypClosed(X),\closedFullRep)}$ is in general not computable, and not even continuous (see e.g.\ \cite[Ex.\ 4.29(viii)]{KechrisCDST}). For a more precise analysis of the complexity of the intersection operator see \cite[Thm.\ 7.1]{BG2008}.

\smallskip

Let us introduce the following notion:

\begin{definition}
	\thlabel{def:computably_compact}
	A compact subset $K$ of a computable metric space $(X,d,\alpha)$ is called \textdef{co-c.e.\ compact} if
	\[ \left\{ \sigma\in\baire \st K\subset \bigcup_{i<\length{\sigma}} B_{\sigma(i)} \right\} \]
	is computably enumerable. We say that $K$ is \textdef{computably compact} if it is co-c.e.\ compact and there exists a computable dense sequence in $K$.
	
	We say that a sequence $\sequence{K_n}{n\in I}$ is \textdef{uniformly co-c.e.\ compact} if each $K_n$ is co-c.e.\ compact in a computable metric space $X_n$ and the set
	\[ \left\{ (n,\sigma)\in\mathbb{N}\times\baire \st K_n \subset \bigcup_{i<\length{\sigma}} B^n_{\sigma(i)} \right\} \]
	is computably enumerable, where $B^n_k$ is the $k$-th basic open ball in $X_n$. In other words, the sequence $\family{K_n}{n\in I}$ is uniformly co-c.e.\ compact if there is a single computable function witnessing that each $K_n$ is co-c.e.\ compact.
\end{definition}

The notions of co-c.e.\ compact and computably compact are standard notions in computable analysis (see e.g.\ \cite[Def.\ 2.10]{BGP17}).
Notice that being co-c.e.\ compact implies being $\lightfacePi^0_1(X)$ and that every $\lightfacePi^0_1(X)$ subset of a co-c.e.\ compact space is co-c.e.\ compact. Clearly a computable metric space is co-c.e.\ compact iff it is computably compact. Moreover, if $K$ is co-c.e.\ compact (resp.\ computably compact) and $f\function{K}{Y}$ is computable and surjective, then $Y$ is co-c.e.\ compact (resp.\ computably compact) as well (see \cite[Prop.\ 5.3]{Pauly16}). Several equivalent conditions to being computably compact are listed in \cite[Prop.\ 5.2]{Pauly16}. The notions of co-c.e.\ compact and computably compact can be extended in a straightforward way to effective spaces.

We also mention the following simple lemma:

\begin{lemma}
	\thlabel{thm:prod_comp_compact}
	If $X$ is co-c.e.\ compact then so is $\infStrings{X}$.
\end{lemma}
\begin{proof}
	The fact that the finite product of co-c.e.\ compact spaces is co-c.e.\ compact follows from \cite[Prop.\ 5.4]{Pauly16}. To prove that $\infStrings{X}$ is co-c.e.\ compact, recall that an open set in $\infStrings{X}$ is of the type $\mathcal{B}:=\prod_{j\in\mathbb{N}}B_j$, where each $B_j$ is open in $X$ and $B_j\neq X$ only for finitely many indexes. Such an open set is canonically represented via (a name for) a finite sequence $\sequence{B_j}{j<N}$ with the understanding that $B_j=X$ for every $j\ge N$.
	
	Let $\sequence{\mathcal{B}_i}{i<k}$ be a finite sequence of open subsets of $\infStrings{X}$, where $\mathcal{B}_i$ is represented by $\sequence{B^i_j}{j<N_i}$. This sequence trivially induces a finite sequence $\sequence{\mathcal{C}_i}{i<k}$ of open subsets of $X^N$, where $N:=\max_{i<k} N_i$: for every $i$ and every $j\in\{ N_i,\hdots,N-1 \}$, let $B^i_j:=X$ and define $\mathcal{C}_i:=\prod_{j<N} B^i_j$. The sequence $\sequence{\mathcal{B}_i}{i<k}$ covers $\infStrings{X}$ iff $\sequence{\mathcal{C}_i}{i<k}$ covers $X^N$. The claim follows from the fact that the sets $\sequence{X^n}{n\in\mathbb{N}}$ are uniformly co-c.e.\ compact.
\end{proof}

Since $\hypCompact(X)\subset \hypClosed(X)$, we can represent the compact subsets of $X$ using the subspace representation induced by the representation we put on $\hypClosed(X)$. At the same time, we mentioned that $\hypCompactV(X)$ is compatible with the Hausdorff metric. Letting $\beta$ be an enumeration of the finite subsets of $\ran(\alpha)$, the space $(\hypCompactV(X),\hmetric,\beta)$ is a computable metric space (as the finite subsets of $\ran(\alpha)$ are a dense subset of $\hypCompactV(X)$). In particular, $\hypCompactV(X)$ can be (canonically) endowed with the Cauchy representation $\cauchyRep$ induced by $\hmetric$ and $\beta$. The Cauchy representation for the hyperspace of non-empty compact subset of the Euclidean space was studied in \cite{BW1999} under the name $\delta_{\mathrm{Haus}}$, and then extended to generic computable metric spaces in \cite{BrattkaPresser2003}, where the symbol $\repmap{\mathrm{Hausdorff}}$ was used.

Two additional representations maps for $\hypCompact(X)$ that are used in the literature (see e.g.\ \cite[Sec.\ 4]{BdBPLow12}) are the maps $\compactCover$ and $\compactMinCover$: the former names a compact set $K$ via a list of all finite covers of $K$ with basic open balls, while $\compactMinCover$-names have the additional requirement that all basic balls in each cover have to intersect $K$. The compact sets having computable $\compactCover$-names (resp.\ $\compactMinCover$-names) are exactly the co-c.e.\ compact (resp.\ computably compact) sets. The representation maps $\compactCover$ and $\compactMinCover$ have been studied in \cite{BrattkaPresser2003,BW1999} under the names $\delta_{\mathrm{cover}}$ and $\delta_{\mathrm{min-cover}}$ respectively.

\begin{proposition}
	Let $X$ be a co-c.e.\ compact subspace of the Euclidean space $(\mathbb{R}^d,|\cdot|, \alpha_{\mathbb{Q}^d})$, where $|\cdot|$ is the Euclidean distance and $\alpha_{\mathbb{Q}^d}$ is an enumeration of $\mathbb{Q}^d$. Let $\closedNegRep':=\closedNegRep\restrict{\closedNegRep^{-1}(\hypCompact(X))}$ be the restriction of $\closedNegRep$ to names of compact subsets of $X$, and similarly let $\closedFullRep':=\closedFullRep\restrict{\closedFullRep^{-1}(\hypCompact(X))}$. The following holds on $\hypCompact(X)$:
	\begin{itemize}
		\item $\compactCover \equiv \closedNegRep'$;
		\item $\compactMinCover \equiv \closedFullRep' \equiv \cauchyRep$.
	\end{itemize}
\end{proposition}
\begin{proof}
	The equivalence $\compactCover \equiv \closedNegRep'$ follows from \cite[Thm.\ 4.6]{BW1999}. Precisely, the representation $\delta^>_{\mathscr{K}}$ names a compact set $K$ via a $\closedNegRep'$-name of $K$ and the index of a ball $B$ s.t.\ $K\subset \closure{B}$ (\cite[Def.\ 4.1]{BW1999}). If $X$ is compact then the equivalence $\closedNegRep'\equiv \delta^>_{\mathscr{K}}$ is straightforward.
	
	The equivalence $\compactMinCover \equiv \closedFullRep'$ follows from \cite[Cor.\ 4.7]{BW1999} using the fact that $\closedFullRep' \equiv \delta^=_{\mathscr{K}}$ for a compact $X$. Indeed, a $\delta^=_{\mathscr{K}}$-name of a compact set $K$ is a $\closedFullRep'$-name of $K$ and the index of a ball $B$ s.t.\ $K\subset \closure{B}$ (\cite[Def.\ 4.1]{BW1999}).
	
	In \cite[Thm.\ 4.10]{BW1999}, the Cauchy representation $\delta_{\mathrm{Haus}}$ (i.e.\ the restriction of $\cauchyRep$ to names of non-empty compact sets) has been proven equivalent to the restriction $\delta^=_{\mathscr{K}}|^{\mathscr{K}^*}$ of $\delta^=_{\mathscr{K}}$ to names of non-empty compact sets. We now explicitly show that $\cauchyRep\equiv \closedFullRep'$ (i.e.\ the empty set is not problematic if $X$ is co-c.e.\ compact).
	
	Recall that for every non-empty closed $G$, $\hmetric(G,\emptyset)=1$.
	\begin{description}
		\item[$\cauchyRep\le \closedFullRep'$]: As mentioned after the definition of Cauchy representation, we can think of a $\cauchyRep$-name $p$ for $F$ as a list of (indexes for) basic open balls $\sequence{\mathcal{B}_n}{n\in\mathbb{N}}$ w.r.t.\ the Hausdorff metric s.t.\ all the balls contain $F$ and the radius of $\mathcal{B}_n$ is $2^{-n}$.
		Since the empty set is isolated in $\hypCompactV(X)$, it is enough to consider $n>1$. 
		Indeed, without loss of generality we can assume that we can computably tell whether $\mathcal{B}_n$ is centered on $\emptyset$ (e.g.\ we can assume that the empty set has a unique index in the list of dense subsets of $\hypCompact(X)$).
		If $\mathcal{B}_n$ is centered on $\emptyset$ then $F=\emptyset$ (as $\mathcal{B}_n=\{\emptyset\}$). Otherwise $F\neq\emptyset$, and we can use $\delta_{\mathrm{Haus}}\le \delta^=_{\mathscr{K}}|^{\mathscr{K}^*}$ and $\delta^=_{\mathscr{K}} \le \closedFullRep'$.
		\item[$\closedFullRep'\le \cauchyRep$]: Let $\coding{p,q}$ be a $\closedFullRep'$-name for $F$, where $p$ is a negative information name and $q$ is a positive information name for $F$. If $F=\emptyset$ then $p$ is a list of basic open balls that cover $X$. On the other hand, if $F\neq\emptyset$ then $q$ eventually lists some basic open ball (in $X$) that intersects $F$ ($q$ is allowed not to produce any information at a given stage). In other words, we wait for some sufficiently large $n$ so that either $\bigcup_{i<n} B_{p(i)}$ covers $X$ or $q$ commits to some open ball at stage $n$. This allows us to determine whether $F$ is empty or not. If $F=\emptyset$ we can trivially compute a sequence of basic open balls (in $\hypCompactV(X)$) centered on $\emptyset$ with rapidly decreasing radii. Otherwise, as in the previous reduction, we can use $\closedFullRep' \le \delta^=_{\mathscr{K}}$ and $\delta^=_{\mathscr{K}}|^{\mathscr{K}^*} \le \delta_{\mathrm{Haus}}$ to produce a $\cauchyRep$-name for $F$. \qedhere
	\end{description}

\end{proof}

For the sake of readability, we adopt the following notation:
\begin{description}
	\item $\hypClosedUF(X)=(\hypClosed(X),\tau_{UF}, \closedNegRep)$ is the hyperspace of closed subsets of $X$, endowed with the upper Fell topology and the negative representation $\closedNegRep$;
	\item $\hypClosedF(X)=(\hypClosed(X),\tau_{F}, \closedFullRep)$ is the hyperspace of closed subsets of $X$, endowed with the Fell topology and the full representation $\closedFullRep$;
	\item $\hypCompactUF(X)=(\hypCompact(X),\tau_{UF}\restrict{\hypCompact(X)}, \closedNegRep')$ is the hyperspace of compact subsets of $X$, endowed with the upper Fell topology and the negative representation restricted to names of compact sets, i.e.\ $\closedNegRep'=\closedNegRep\restrict{\closedNegRep^{-1}(\hypCompact(X))}$;
	\item $\hypCompactV(X)=(\hypCompact(X),\tau_{V}\restrict{\hypCompact(X)}, \closedFullRep')$ is the hyperspace of compact subsets of $X$, endowed with the Vietoris topology and the full representation restricted to names of compact sets, i.e.\ $\closedFullRep'=\closedFullRep\restrict{\closedFullRep^{-1}(\hypCompact(X))}$.
\end{description}

The following lemma is the effective counterpart of \cite[Thm.\ 4.26]{KechrisCDST}.

\begin{lemma}
	\thlabel{thm:K(X)_computably_compact}
	If $(X,d,\alpha)$ is computably compact then so is $\hypCompactV(X)$.
\end{lemma}
\begin{proof}
	Recall that a basic open ball in $\hypCompactV(X)$ is a set of the type $\mathcal{B}_{\coding{n,m}} = \{K\in\hypCompact(X) \st \hmetric(K,\beta(n))< q_m\}$, where $\beta$ is the fixed  enumeration of the finite subsets of $\ran(\alpha)$ and $\sequence{q_i}{i\in\mathbb{N}}$ is the canonical enumeration of $\mathbb{Q}^+$. In particular, if $\beta(n)=\{b_0,\hdots, b_{k-1} \}$ then \[ \mathcal{B}_{\coding{n,m}}=\left\{K\in\hypCompact(X) \st K\subset \bigcup_{i<k} \ball{b_i}{q_m} \land (\forall i<k)(K\cap \ball{b_i}{q_m}\neq \emptyset) \right\}. \]

Let $\sequence{\mathcal{B}_i=\mathcal{B}_{\coding{n_i,m_i}}}{i<k}$ be a finite sequence of basic open balls in $\hypCompactV(X)$. We want to describe a c.e.\ procedure to check whether $\sequence{\mathcal{B}_i}{i<k}$ covers $\hypCompactV(X)$. For the sake of readability, let us define $r(i):=q_{m_i}$
and $D_i:=\bigcup_{b\in \beta(n_i)} \ball{b}{r(i)}$ for every $i<k$. We first check if $r(i)>1$ for some $i<k$ (any ball with radius $>1$ covers $\hypCompactV(X)$, and in this case we can give a positive answer) or if there is $j<k$ s.t.\ $\beta(n_j)=\emptyset$ (recall that the empty set is
isolated, and a ball with radius $\le 1$ covers $\emptyset$ iff it is centered on it, hence if such a $j$ is missing we give a negative answer). Since these two conditions are computable, to prove the result it is enough to show that there is a c.e.\ procedure to determine if a finite sequence
$\sequence{\mathcal{B}_i}{i<k}$ with $r(i) \le 1$ and $\beta(n_i) \neq \emptyset$ for all $i<k$ covers $\hypCompactV(X)\setminus \{\emptyset\}$.
	
We now define a c.e.\ subtree $T$ of $(\mathbb{N}\times\ran(\alpha))^{\le k}$ (where $k$ is the length of the fixed finite sequence of basic open balls) where each string $\sigma\in T$ is labeled with a compact subset $Y_\sigma$ of $X$. A string $\sigma=\str{(a_0,b_0),\hdots, (a_{h-1},b_{h-1})}$ is in $T$ iff
\[ h\le k \land (\forall i<h)(a_i<k \land b_i\in \beta(n_{a_i})) \land (\forall i,j<h, i\neq j)(a_i\neq a_j)\land Y_{\sigma[h-1]}\subset D_{a_{h-1}}, \]
where $\sigma[h-1]$ denotes the prefix of $\sigma$ of length $h-1$. The string $\sigma$ is labeled with $Y_\sigma:= X\setminus \bigcup_{i<h} \ball{b_i}{r(a_i)}$.
	
Notice that, since $X$ is computably compact, then so are all the sets $Y_\sigma$. In fact, they are uniformly computably compact: given a finite sequence $\sequence{B_i}{i<n}$ of open balls, we can compute the index of a computable functional witnessing that $Y:=X\setminus \bigcup_{i<n} B_i$ is
computably compact. This follows from 
\[ Y \subset \bigcup_{j<\length{\tau}} B_{\tau(j)} \iff X \subset  \bigcup_{i<n} B_i \cup \bigcup_{j<\length{\tau}} B_{\tau(j)}~. \]
In particular, this shows that the condition $Y_{\sigma[h-1]}\subset D_{a_{h-1}}$ (and hence $\sigma\in T$) is computably enumerable.

	We claim that
	\[ \sequence{\mathcal{B}_i}{i<k} \text{ covers }\hypCompactV(X)\setminus \{\emptyset\} \iff \text{ for each leaf }\sigma\in T,~ Y_\sigma = \emptyset. \]
This implies the computable compactness of $\hypCompactV(X)$, as the condition on the right is computably enumerable ($Y_\sigma=\emptyset$ is equivalent to $Y_\sigma$ being covered by the empty set and the quantification on $\sigma$ is bounded because $T$ is computably finitely branching).
	
	For the right-to-left direction, assume that there is a non-empty $K\in \hypCompact(X)\setminus \bigcup_{i<k} \mathcal{B}_i$. We show that there is a leaf $\sigma$ s.t.\ $Y_\sigma\neq\emptyset$. We proceed iteratively as follows, defining a list of sequences $\sigma_s$ s.t.\ $K\subset Y_{\sigma_s}$: let $\sigma_{-1}:=\str{}$. 
	At stage $s<k$, we look for some unmarked (i.e.\ not selected in any previous stage) $a_s<k$ s.t.\ $Y_{\sigma_{s-1}}\subset D_{a_s}$. If such a choice is possible then, since $K\notin \mathcal{B}_{a_s}$ and $K\subset D_{a_s}$, there is $b_s\in \beta(n_{a_s})$ s.t.\ $K\cap \ball{b_s}{r(a_s)}=\emptyset$. In particular, $K\subset Y_{\sigma_{s-1}} \setminus \ball{b_s}{r(a_s)}$, hence letting $\sigma_s:=\sigma_{s-1}\concat\str{(a_s,b_s)}$, we have $K\subset Y_{\sigma_s}$, which implies that $Y_{\sigma_s}\neq\emptyset$. We then mark $a_s$ as visited and go to the next stage. If there is no $a_s<k$ that satisfies the requirements then $\sigma:=\sigma_{s-1}$ is a leaf for $T$ but $Y_{\sigma}\neq\emptyset$ (as $K\subset Y_\sigma$). Notice that looking for some unmarked $a_s$ guarantees that no choice of $a_s$ is possible at stage $k+1$.

    Let us now prove the left-to-right implication. Assume that $\sigma=\str{(a_0,b_0),\hdots, (a_{h-1},b_{h-1})}$ is a leaf for $T$. Notice that, for each $j<h$ and each $K\in\hypCompact(X)\setminus \{\emptyset \}$, if $K\in \mathcal{B}_{a_j}$ then for every $b\in \beta(n_{a_j})$, $K\cap
    \ball{b}{r(a_j)}\neq\emptyset$, and hence $K\not\subset Y_{\sigma[j+1]}$. In particular, if $Y_{\sigma}\neq\emptyset$ then, for every $j<h$, $Y_{\sigma}\notin \mathcal{B}_{a_j}$. Moreover, since $\sigma$ is a leaf, there is no pair $(a,b)$ s.t.\ $Y_\sigma \subset D_a$, which implies that $Y_\sigma \notin \bigcup_{i<k} \mathcal{B}_i$.
\end{proof}

\begin{proposition}
	\thlabel{thm:hypclosed_Rd_eff_compact}
	$\hypClosedF(\mathbb{R}^d)$ is computably compact.
\end{proposition}
\begin{proof}
	We show that there is a computable surjection $f\function{\hypCompactV([0,1]^d)}{\hypClosedF(\mathbb{R}^d)}$, and the claim will follow using \thref{thm:K(X)_computably_compact} and the remarks following \thref{def:computably_compact}. Fix a computable homeomorphism $\varphi\function{(0,1)^d}{\mathbb{R}^d}$ and define
	\[ f(K):=\varphi(K\cap (0,1)^d)~. \]
	It is easy to see that, for $K\in\hypCompact([0,1]^d)$, $f(K)$ is closed. 
	Moreover, $f$ is surjective: for every $F\in \hypClosed(\mathbb{R}^d)$, $\varphi^{-1}(F)$ is closed in the relative topology of $(0,1)^d$. If we denote with $G$ its closure w.r.t.\ the relative topology of $[0,1]^d$, we have that $G\setdifference \varphi^{-1}(F) \subset \boundary([0,1]^d)$, hence, in particular, $f(G)=F$.
	
	Finally, we show that $f$ is computable. Recall that both $\hypCompactV([0,1]^d)$ and $\hypClosedF(\mathbb{R}^d)$ are admissibly represented with the full information representation. Let $\coding{p,q}\in \Baire$ be a name for $K\in\hypCompact([0,1]^d)$. To compute a negative information name for $f(K)$ from $p$, notice that, since $\varphi^{-1}$ is a computable homeomorphism (and, in particular, it is a total computable surjection), as mentioned in the proof of \cite[Prop.\ 3.7]{BdBPLow12} we have that the map $\function*{\hypClosedUF(Y)}{\hypClosedUF(X)}:=A\mapsto \varphi(A)$. Intuitively:  for every basic open $B\subset \setcomplement[{[0,1]^d}]{K}$ we can computably list a sequence of basic open balls of $\mathbb{R}^d$ exhausting $\varphi(B)$.
	On the other hand, notice that a basic open ball $B$ of $\mathbb{R}^d$ intersects $f(K)$ iff there is $i\in\mathbb{N}$ s.t.\ $B_{q(i)}\subset \varphi^{-1}(B)$ (this follows from the fact that $\varphi$ is a homeomorphism). In particular, to produce a positive information name for $f(K)$, we list $B_n\subset \mathbb{R}^d$ whenever we find some $i$ s.t.\ $B_{q(i)}\subset \varphi^{-1}(B_n)$ (which is a computable condition).
\end{proof}

Recall that, if $X$ is not compact, then the hyperspace $(\hypClosed(X),\tau_V)$ of closed subsets $X$ endowed with the Vietoris topology is not metrizable. We now show that it is not even admissibly represented.

\begin{proposition}
	\thlabel{thm:upper_vietoris_no_pseudobase}
	Let $\hypClosedUV(X)=(\hypClosed(X),\tau_{UV})$ denote the hyperspace of closed subsets of $X$ endowed with the upper Vietoris topology $\tau_{UV}$. The space $\hypClosedUV(\mathbb{R})$ (and hence $\hypClosedUV(\mathbb{R}^d)$) does not have a countable pseudobase. In particular, it is not second-countable and it is not admissibly represented.
\end{proposition}
\begin{proof}
Fix a countable sequence $\family{\mathcal{P}_i}{i\in\mathbb{N}}$ of subsets
of $\hypClosedUV(\mathbb{R})$. To show that
$\family{\mathcal{P}_i}{i\in\mathbb{N}}$ is not a pseudobase we build a
closed set $F$ and an open set $\mathcal{U}$ which contains $F$ s.t.\ for
every $i$, either $F\notin \mathcal{P}_i$ or $\mathcal{P}_i \not\subset
\mathcal{U}$. We define $F:=\{x_i\st i \in\mathbb{N} \}$, where
$\sequence{x_i}{i\in\mathbb{N}}$ is a strictly increasing sequence
iteratively defined as follows: for each $i$, let $n_i$ be the smallest
integer greater than $x_j+1$ for every $j<i$ (if $i=0$ we let $n_i:=0$).
Choose an unbounded $P_i\in \mathcal{P}_i$. If there is none we just define
$x_i:=n_i$. Let $y_i \in P_i \cap [n_i, \infty)$ and choose $x_i$ s.t.\
$x_i>y_i+1$. In particular $d(y_i,\family{x_j}{j\le i})>1$.

	Notice that, for every $i\neq j$, $d(x_i,x_j)>1$, hence the sequence $\sequence{x_i}{i\in\mathbb{N}}$ does not have accumulation points. In particular, the set $F:=\{ x_i \st i \in \mathbb{N}\}$ is closed and unbounded. Fix $\varepsilon$ sufficiently small, e.g.\ $\varepsilon = 1/4$, and define the open set $F_\varepsilon := \{ x\in\mathbb{R} \st d(x,F)<\varepsilon \}$ and $\mathcal{U}:=\{ G\in \hypClosed(\mathbb{R})\st G\subset F_\varepsilon \}$.

	The open set $\mathcal{U}$ and the closed set $F$ witness the fact that $\family{\mathcal{P}_i}{i\in\mathbb{N}}$ is not a pseudobase. Indeed, for every $i$, either every $P\in \mathcal{P}_i$ is bounded (and hence $F\notin \mathcal{P}_i$), or the set $P_i$ defined above witnesses that $\mathcal{P}_i\not\subset \mathcal{U}$ (as by construction $d(y_i,F)>1$).

	This implies also that $\hypClosedUV(\mathbb{R})$ is not second-countable, as every base is a pseudobase. The fact that it is not admissibly represented follows by \thref{thm:characterization_admissibility}. The claim generalizes to $\hypClosedUV(\mathbb{R}^d)$ as every pseudobase of $\hypClosedUV(\mathbb{R}^d)$ induces a pseudobase on $\hypClosedUV(\mathbb{R})$ by projection.
\end{proof}

\begin{corollary}
	The space $(\hypClosed(\mathbb{R}),\tau_V)$ (and hence $(\hypClosed(\mathbb{R}^d),\tau_V)$) does not have a countable pseudobase. In particular, it is not second-countable and it is not admissibly represented.
\end{corollary}
\begin{proof}
	This follows from the proof of \thref{thm:upper_vietoris_no_pseudobase}. Indeed,  the above proof only uses a closed set $F\in \hypClosed(\mathbb{R})$ and an open set $\mathcal{U}\subset \hypClosedUV(\mathbb{R})$ to show that no countable subfamily of $\hypClosed(\mathbb{R})$ is a pseudobase. Since the upper Vietoris topology is coarser than the Vietoris topology, the same argument applies to $(\hypClosed(\mathbb{R}),\tau_V)$. The claim would not follow immediately if, in the proof of \thref{thm:upper_vietoris_no_pseudobase} we would have exploited a convergent sequence to $F$, as convergence is a weaker notion in the upper Vietoris topology.
\end{proof}

We conclude this section with the effective counterpart of \cite[Lem.\ 1.3]{AndrMarcODE97}:

\begin{lemma}
	\thlabel{thm:effective_AM}
	Let $X,Y$ be computable metric spaces. If $Y$ is co-c.e.\ compact then, for every $F\in \lightfacePi^0_1(X\times Y)$, $\proj_X F \in \lightfacePi^0_1(X)$. If $Y=\bigcup_{n\in\mathbb{N}} Y_n$ where the sequence $\sequence{Y_n}{n\in\mathbb{N}}$ is uniformly co-c.e.\ compact, then for every $F\in \lightfaceSigma^0_2(X\times Y)$, $\proj_X F \in \lightfaceSigma^0_2(X)$.
\end{lemma}
\begin{proof}
	Let us first assume that $Y$ is co-c.e.\ compact and let $F\in \lightfacePi^0_1(X\times Y)$. Let $p\in\Baire$ be a computable map s.t.\ $\setcomplement{F}=\bigcup_{n\in\mathbb{N}} B^X_{p(n)_0}\times B^Y_{p(n)_1}$. Notice that $B^X_n\subset \setcomplement{(\proj_X F)}$ iff the preimage $B^X_n \times Y$ of $B^X_n$ via the projection map $\proj_X$ is contained in the complement of $F$.
	
	Let $\varphi_p\function{\baire\times \mathbb{N}}{\mathbb{N}}$ be a computable function s.t., for all $\sigma\in \baire$
	\[ \bigcap_{i\in\ran(\sigma)} B^X_{p(i)_0} = \bigcup_{k\in\mathbb{N}} B^X_{\varphi_p(\sigma,k)}~. \]
	Such a map exists because $X$ is a computable metric space (and hence effective second-countable).
	
	To show that $\setcomplement{(\proj_X F)}$ is effectively open, notice that
	\[ \left\{ n \in \mathbb{N} \st (\exists \sigma \in \baire)\left(
\bigcup_{i\in\ran(\sigma)} B^Y_{p(i)_1}= Y \text{ and } (\exists
k)(n=\varphi_p(\sigma,k)) \right) \right\} \in \lightfaceSigma^0_1. \] This
follows from the fact that $\varphi_p$ is computable and that
$\bigcup_{i\in\ran(\sigma)} B^Y_{p(i)_1}= Y $ is $\lightfaceSigma^0_1$
because $Y$ is co-c.e.\ compact. This shows that we can computably enumerate
a list of open sets exhausting the complement of $\proj_X F$, i.e.\ $\proj_X
F \in \lightfacePi^0_1(X)$.
	
The same argument shows (still assuming $Y$ co-c.e.\ compact) that if $D\in \lightfacePi^0_1(\mathbb{N}\times X \times Y)$ then $\proj_{\mathbb{N}\times X} D \in \lightfacePi^0_1(\mathbb{N}\times X)$ (it is enough to replace $X$ with $\mathbb{N}\times X$). If $F\in \lightfaceSigma^0_2(X\times Y)$ then it can be written as $F= \proj_{X\times Y} D$, for some $D\in \lightfacePi^0_1(\mathbb{N}\times X \times Y)$. In particular,
	\[ \proj_X F = \proj_X \proj_{\mathbb{N}\times X} D~,  \]
	and therefore $\proj_X F \in \lightfaceSigma^0_2(X)$.
	
	Finally, assume that $F\in \lightfaceSigma^0_2(X\times Y)$ and $Y=\bigcup_{n\in\mathbb{N}} Y_n$ where the $\sequence{Y_n}{n\in\mathbb{N}}$ are uniformly co-c.e.\ compact. Let $D\in\lightfacePi^0_1(\mathbb{N}\times X \times Y)$ be s.t.\ $F=\proj_{X\times Y} D$ (as above).
	Notice that, defining $D_n:=\{ (k,x,y) \in D \st y \in Y_n \}$, the sequence $\sequence{\proj_{\mathbb{N}\times X} D_n}{n\in\mathbb{N}}$ is uniformly $\lightfacePi^0_1$, as $\sequence{Y_n}{n\in\mathbb{N}}$ is uniformly co-c.e.\ compact. In other words,
	\[ E:= \{ (n,k,x) \st (\exists y)((k,x,y)\in D_n )\}\in \lightfacePi^0_1(\mathbb{N}\times \mathbb{N}\times X), \]
	and therefore $\proj_X F = \proj_X E\in \lightfaceSigma^0_2(X)$ (using a canonical computable identification $\mathbb{N}^2\to \mathbb{N}$).
\end{proof}

We cannot extend \thref{thm:effective_AM} to effective spaces because in that
context $\lightfaceSigma^0_2$ sets are not necessarily computable unions of
$\lightfacePi^0_1$ sets. However the above proof works if we deal only with
subsets of the product space that are computable unions of $\lightfacePi^0_1$
sets. We thus obtain the following Corollary.

\begin{corollary}
\thlabel{thm:effective_AM_eff_spaces} Let $(X,
\sequence{B^X_n}{n\in\mathbb{N}}),(Y,\sequence{B^Y_n}{n\in\mathbb{N}})$ be
effective second-countable spaces. If $Y$ is co-c.e.\ compact then, for every
$F\in \lightfacePi^0_1(X\times Y)$, $\proj_X F \in \lightfacePi^0_1(X)$. If
$Y=\bigcup_{n\in\mathbb{N}} Y_n$, where the sequence
$\sequence{Y_n}{n\in\mathbb{N}}$ is uniformly co-c.e.\ compact, and
$F=\proj_{X\times Y} D$ for some $D \in \lightfacePi^0_1(\mathbb{N}\times
X\times Y)$, then $\proj_X F \in \lightfaceSigma^0_2(X)$.
\end{corollary}

\section{Computable measure theory}
\label{sec:comp_measure_theory}

If $X$ is a separable metrizable space, we consider the space $\ProbabilityMeas(X)$ of Borel probability measures on $X$, endowed with the weak topology generated by the maps $\mu \mapsto \int f\dd\mu$, with $f\in\cont[][b]{X}$ (i.e.\ $f\colon X \to \mathbb{R}$ is continuous and bounded, see \cite[Sec.\ 17.E]{KechrisCDST}). A basis for the topology on $\ProbabilityMeas(X)$ is the family of sets of the form
\[ U_{\mu,\varepsilon,f_0,\hdots,f_n} := \left\{ \nu\in\ProbabilityMeas(X) \st (\forall i\le n)\left( \left| \int_X f_i \dd{\nu} - \int_X f_i \dd{\mu} \right| < \varepsilon\right) \right\}, \]
where $\mu\in\ProbabilityMeas(X)$, $\varepsilon>0$, and $f_i\in\cont[][b]{X}$ for every $i$. The space $\ProbabilityMeas(X)$ is separable metrizable iff so is $X$ \cite[Ch.\ II, Thm.\ 6.2]{Parthasarathy67}. Moreover if $X$ is compact metrizable (resp.\ Polish) then so is $\ProbabilityMeas(X)$ (\cite[Thm.\ 17.22 and Thm.\ 17.23]{KechrisCDST}).

We now give a brief introduction on how computable measure theory can be developed in the context of TTE. For a more thorough presentation we refer the reader to \cite{CollinsCSP2014}. The theory can be developed in the more general context of Borel measures on sequential topological spaces \cite{Schroder2007}. In particular, since every represented space can be endowed with the final topology (which is sequential), the theory can be developed for every represented space $X$. For our purposes it is enough to focus on probability measures on $X$, where $X$ is (computably homeomorphic to) either $[0,1]^d$ or $\mathbb{R}^d$.

As mentioned, in this case $\ProbabilityMeas(X)$ is a Polish space. A canonical choice for a dense subset of $\ProbabilityMeas(X)$ is the set $\mathcal{D}$ of probability measures concentrated on finitely many points of the dense subset of $X$, assigning rational mass to each of them (i.e.\ a weighted sum of Dirac deltas, where each weight is rational). Moreover, the Prokhorov metric on $\ProbabilityMeas(X)$ can be explicitly defined as
\[ \rho(\mu,\nu) := \inf \{ \varepsilon>0 \st (\forall A \in \Borel(X))(\mu(A) \le \nu(A^\varepsilon)+\varepsilon \text{ and } \nu(A) \le \mu(A^\varepsilon)+\varepsilon)\}, \]
with $A^\varepsilon:=\{ x \in X \st d(x,A)<\varepsilon \}$. This metric induces the weak topology on $\ProbabilityMeas(X)$. The space $(\ProbabilityMeas(X), \rho, \mathcal{D})$ is a computable metric space (see \cite[Prop.\ 4.1.1]{HoyrupRojas2007}), and therefore it is represented using the Cauchy representation.

From a computational point of view, it is often convenient to look at Borel
(probability) measures from a different point of view. A
\textdef{(probability) valuation} is a map
$\nu\function{\boldfaceSigma^0_1(X)}{[0,1]}$ s.t.\
\begin{itemize}
	\item $\nu(\emptyset) = 0$;
	\item $\nu(X) = 1$;
	\item $\nu(U)+\nu(V)= \nu(U \cup V) - \nu(U\cap V)$.
\end{itemize}
Probability valuations can be defined in a slightly more general context as maps over a lattice (\cite[Sec.\ 2.2]{Schroder2007}). Every Borel measure $\mu$ naturally induces a valuation $\nu:=\mu\restrict{\boldfaceSigma^0_1(X)}$. The induced valuation is lower semicontinuous, i.e.\ if $\sequence{A_i}{i\in\mathbb{N}}$ are nested open sets then $\nu(\bigcup_i A_i) = \sup_i \nu(A_i)$. Since every finite Borel measure is uniquely identified by its restriction to the open sets (as every such measure on the Euclidean space is regular, and in particular outer regular, see e.g.\ \cite[Thm.\ 2.18]{RudinRCA}), we can identify $\ProbabilityMeas(X)$ with the family of lower semicontinuous valuations on $\boldfaceSigma^0_1(X)$.

The lower semicontinuity of the valuations can be naturally translated in the context of TTE as follows. We use the represented space $(\leftReal,\repmap{\leftReal})$ of real numbers, where $x\in \mathbb{R}$ is represented by a monotonically increasing sequence of rational numbers converging to $x$. Equivalently, we can think of a $\repmap{\leftReal}$-name for $x$ as the list of all rational numbers smaller than $x$. This is the so-called \textdef{left-cut representation} of the real numbers, and we say that a real is \textdef{left-c.e.\ }if it has a computable $\repmap{\leftReal}$-name (see \cite[Sec.\ 4.1]{Weihrauch00}). The final topology induced on $\mathbb{R}$ by $\repmap{\leftReal}$ is exactly the topology of lower semicontinuity (i.e.\ the topology whose open sets are of the form $(x,\infty)$ for some $x\in\mathbb{R}$, see \cite[Lem.\ 4.1.4]{Weihrauch00}). Similarly, we obtain the represented space $(\rightReal, \repmap{\rightReal})$ of the \textdef{right-c.e.\ }reals, where $\repmap{\rightReal}$ is the \textdef{right-cut representation} map, naming a real as a monotonically decreasing sequence of rationals converging to it. The final topology of $\repmap{\rightReal}$ is the topology of the upper-semicontinuity (again, see \cite[Lem.\ 4.1.4]{Weihrauch00}). It is straightforward to see that given a $\repmap{\leftReal}$-name for $x$ we can computably find a $\repmap{\rightReal}$-name for $-x$ (and vice versa). Notice that $+\function{\leftReal\times \leftReal}{\leftReal}$ and $\sup\function{\leftReal^\mathbb{N}}{\leftReal}$ are computable, but $-\function{\leftReal\times \leftReal}{\leftReal}$ is not. 

With this in mind, we can define another representation on the space of Borel (probability) measures: the \textdef{canonical representation} $\repmap{C}$ for a (probability) measure names a measure $\mu$ using a name for the (realizer-)continuous function $\mu\restrict{\boldfaceSigma^0_1(X)}\function{\boldfaceSigma^0_1(X)}{[0,1]_<}$ (\cite[Sec.\ 3.1]{Schroder2007}). The final topology on $\ProbabilityMeas(X)$ induced by $\repmap{C}$ coincides with the weak topology on $\ProbabilityMeas(X)$ (\cite[Cor.\ 3.5]{Schroder2007}). Moreover, the canonical representation is equivalent to the Cauchy representation on $\ProbabilityMeas(X)$ (\cite[Prop.\ 3.7]{Schroder2007}).
In the development of the theory it is often more convenient to think of a (probability) measure as being represented using the canonical representation, i.e.\ using a name for the induced valuation. We can therefore think of a name for a (probability) measure $\mu$ on $X$ as a list of $\repmap{<}$-names for the measures of the basic open balls.

\begin{theorem}
	\thlabel{thm:computability_measures}
	Let $(X,\repmap{X})$ and $(Y,\repmap{Y})$ be represented spaces, endowed with the final topology induced by their respective representation maps. Let also $\cont{X,Y}$ be the set of continuous functions $\function*{X}{Y}$. The following maps are computable:
	
	\begin{enumerate}
		\item $\function*{\ProbabilityMeas(X)\times \boldfaceSigma^0_1(X)}{\leftReal}:=(\mu, U)\mapsto \mu(U)$;
		\item $\function*{\ProbabilityMeas(X)\times \hypClosedUF(X)}{\rightReal}:=(\mu, F)\mapsto \mu(F);$
		\item $\function*{\ProbabilityMeas(X)\times \boldfaceDelta^0_1(X)}{\mathbb{R}}:=(\mu, D)\mapsto \mu(D);$
		\item $\function*{\ProbabilityMeas(X)\times \cont{X,Y}}{\ProbabilityMeas(Y)}:=(\mu, f)\mapsto \mu_f$, where $\mu_f(E):= \mu(f^{-1}(E))$ is the push-forward measure;
		\item $\int\function{\cont{X,\leftReal}\times \ProbabilityMeas(X)}{\leftReal}:= (f,\mu)\mapsto \int f \dd{\mu}$;
		\item $\int\function{\cont[][ebd]{X,\mathbb{R}}\times \ProbabilityMeas(X)}{\mathbb{R}}:= (f,\mu)\mapsto \int f \dd{\mu}$, where $\cont[][ebd]{X,\mathbb{R}}$ denotes the space of \textdef{effectively bounded continuous functions}, i.e.\ $f\in\cont{X,\mathbb{R}}$ and there are two computable reals $a,b$ s.t.\ for every $x\in X$, $a<f(x)<b$.
	\end{enumerate}
\end{theorem}
\begin{proof}
	Point $(1)$ is straightforward from the definition of the canonical representation for $\ProbabilityMeas(X)$ (see also \cite[Prop.\ 4.2.1]{HoyrupRojas2007}) and point $(2)$ is a corollary of point $(1)$ (as $\mu(F) = 1 - \mu(\setcomplement[X]{F})$). Point $(3)$ follows trivially from the points $(1)$ and $(2)$ as a $\repmap{\mathbb{R}}$-name for $x\in \mathbb{R}$ can be computably obtained from a $\repmap{\leftReal}$-name and a $\repmap{\rightReal}$-name of $x$. Point $(4)$ is (essentially) a diagram-chasing exercise, see also \cite[Prop.\ 49]{CollinsCSP2014}. Points $(5)$ and $(6)$ are presented in \cite[Sec.\ 3, in particular point (5) is Prop.\ 7]{CollinsCSP2014}. See also \cite[Prop.\ 3.6]{Schroder2007} for a slightly more general version of point $(5)$.
\end{proof}

For our purposes, we will also need an effective analogue of the fact that if $X$ is compact metrizable then so is $\ProbabilityMeas(X)$ (\cite[Thm.\ 17.22]{KechrisCDST}).  The proof of the following theorem was suggested by Matthias Schr\"oder.

\begin{theorem}
	\thlabel{thm:P(X)_comp_compact}
	For every computable metric space $(X,d,\alpha)$, if $X$ is computably compact then so is $\ProbabilityMeas(X)$.
\end{theorem}
\begin{proof}
	Since $X$ is a computably compact computable metric space, there is a total representation map $\repmap{}\function{\Cantor}{X}$ for $X$ which is (computably) equivalent to the Cauchy representation on $X$ (\cite[Prop.\ 4.1]{BdBPLow12}).
	
	Every probability measure $\mu\in \ProbabilityMeas(\Cantor)$ can be identified with a function $\pi_\mu\in [0,1]^\mathbb{N}$ (identifying $\mathbb{N}$ with $\cantor$) s.t.
	\begin{equation*}\tag{$\star$}
		\pi_\mu(\str{})=1 \land (\forall \sigma \in \cantor)\left( \pi_\mu(\sigma) = \pi_\mu(\sigma\concat \str{0}) + \pi_\mu(\sigma\concat \str{1})\right),
	\end{equation*}	
	so that $\pi_\mu(\sigma)=\mu(\sigma \concat \Cantor)$.
	
	The map $\Phi:=\mu\mapsto \pi_\mu$ is a computable homeomorphism (i.e.\ a computable bijection with computable inverse) between $\ProbabilityMeas(\Cantor)$ and a $\lightfacePi^0_1$ subset of $\infStrings{[0,1]}$. Indeed, it is computable by \thref{thm:computability_measures}(3) and its inverse $\Phi^{-1}$ is straightforwardly computable. Moreover, $\pi \in \ran(\Phi)$ iff it satisfies $(\star)$, which is a $\lightfacePi^0_1$ condition relative to $\pi$. This, in turn, implies that $\ran(\Phi)$ is co-c.e.\ compact by \thref{thm:prod_comp_compact} and the remarks preceding it, as $\infStrings{[0,1]}$ is computably compact. In particular, the fact that $\ProbabilityMeas(\Cantor)$ is computably homeomorphic to a co-c.e.\ compact space implies it is co-c.e.\ compact.
	
	To conclude the proof, define $\psi\function{\ProbabilityMeas(\Cantor)}{\ProbabilityMeas(X)}$ as the push-forward operator $\psi(\mu):=\mu_{\repmap{}}$ where $\mu_{\repmap{}}(E):=\mu(\repmap{}^{-1}(E))$. This map is computable (\thref{thm:computability_measures}(4)) and surjective (\cite[Thm.\ 14]{SchroderSimpson2006}), and therefore $\ProbabilityMeas(X)$ is co-c.e.\ compact.
\end{proof}

\begin{corollary}
	\thlabel{thm:P(X)_comp_compact_k}
	For every computable metric space $X$ that admits an admissible representation $\repmap{}\pfunction{\Baire}{X}$ with co-c.e.\ compact domain, the space $\ProbabilityMeas(X)$ is co-c.e.\ compact.
\end{corollary}
\begin{proof}
	Trivial from \thref{thm:P(X)_comp_compact} since if $X$ admits an admissible representation $\repmap{}\pfunction{\Baire}{X}$ with co-c.e.\ compact domain then it is computably compact.
\end{proof}

\begin{corollary}
	\thlabel{thm:P(K)_comp_compact}
	For every $d$ and every co-c.e.\ compact $K\subset\mathbb{R}^d$, the space $\ProbabilityMeas(K)$ is co-c.e.\ compact, and hence computably compact.
\end{corollary}
\begin{proof}
	This follows from the fact that the interval $[-1,1]$ is admissibly represented by the signed-digit representation
	\[ \repmap{S}\function{\infStrings{3}}{[-1,1]}:=p \mapsto \sum_{i\in\mathbb{N}} (p(i)-1)2^{-i-1} ~,\]
	see e.g.\ \cite[Sec.\ 7.2]{Weihrauch00}. This, in turn, implies that for each $d,n\in \mathbb{N}$ the cube $[-n,n]^d$ is admissibly represented by a total $\infStrings{3}$-representation map $\repmap{d,n}$. Clearly, for every compact $K\subset \mathbb{R}^d$ there is $n$ s.t.\ $K\subset [-n,n]^d$. The restriction of $\repmap{d,n}$ to $\repmap{d,n}^{-1}(K)$ is an admissible representation with co-c.e.\ compact domain, hence the claim follows by \thref{thm:P(X)_comp_compact_k}.
\end{proof}

\section{The effective Kaufman theorem}
\label{sec:eff_kaufman}
In the introduction, we mentioned that one of the first explicit examples of Salem sets is the set $E(\alpha)$ of $\alpha$-well approximable numbers. More formally, for every $\alpha\ge 0$, we define
\[ E(\alpha):=\left\{ x \in [0,1] \st (\exists^\infty n)\left(\min_{m\in\mathbb{Z}} | nx - m | \le n^{-1-\alpha}\right) \right\}. \]
The set $E(\alpha)$ is Salem with dimension $2/(2+\alpha)$ (if $\alpha=0$ then, by Dirichlet's theorem \cite[Ex.\ 10.8]{Falconer14}, $E(\alpha)=[0,1]$). Notice that $E(\alpha)$ is a $\boldfacePi^0_3$ subset of $[0,1]$, as it can be written in the form
\[ \bigcap_{k\in\mathbb{N}} \bigcup_{n \ge k} D_n,\]
where $D_n:=\{ x\in [0,1] \st \min_{m\in\mathbb{Z}} | nx - m | \le n^{-1-\alpha} \}$ is a finite union of non-degenerate closed intervals, and hence is closed. In fact, for every $\alpha>0$, $E(\alpha)$ is not closed (as it is dense in $[0,1]$ but $\hdim(E(\alpha))<1$).

In his original proof of the fact that $E(\alpha)$ is Salem, Kaufman \cite{Kaufman81} defines a measure supported on a (compact) subset of $E(\alpha)$ witnessing that $\fourierdim(E(\alpha))\ge p$ for every $p\le 2/(2+\alpha)$. A similar strategy is adopted in \cite[Ch.\ 9]{Wolff03}, where the author defines a closed subset $S(\alpha)$ of $E(\alpha)$ with $\fourierdim(S(\alpha))=2/(2+\alpha)$. This set can be written as
\[ S(\alpha) = \bigcap_{k\in\mathbb{N}} \bigcup_{n\in P(\alpha,k)} D_n~, \]
where $P(\alpha,k)\subset \mathbb{N}$ is finite. There is some freedom in the precise choice of $P(\alpha,k)$, but in any case the set $S(\alpha)$ is closed as the inner union is now finite.

We now prove an effective analogue of Kaufman theorem, i.e.\ we show that the
map $(\alpha,k)\mapsto P(\alpha,k)$ can be chosen to be computable. This, in
turn, implies that the map $\alpha\mapsto S(\alpha)$ is computable as well.
To this end, we follow the blueprint of the proof strategy presented in
\cite[Ch.\ 9]{Wolff03}, making the estimates explicit so to obtain their
(uniform) computability. We also change a few details to fix a minor
imprecision. The estimates we present are often not tight, and it is
certainly possible to have more precise bounds while retaining the (uniform)
computability.

We start with a couple of technical lemmas. Let us mention that the name for a smooth function $f\function{\mathbb{R}}{\mathbb{R}}$ with compact support is a $\compactMinCover$-name for $\support(f)$ and a sequence of $\cont[0][]{}(\mathbb{R},\mathbb{R})$-names for the $n$-th derivative $f^{(n)}$ of $f$.

Let us denote with $\torus$ the torus (i.e.\ the interval $[0,1]$ with endpoints identified). In particular, $1$-periodic smooth functions with support in $[0,1]$ can be identified with smooth functions on $\torus$.

\begin{lemma}
	\thlabel{thm:schwartz_decay}
	If $f$ is a smooth function on $\torus$ then for every $N$ and every $x$
	\[ |\fouriertransform{f}(x)| \le \frac{\eta_{f,N}}{(1+|x|)^N} \]
	where $\eta_{f,N}$ is uniformly computable from $N$ and $f$.
\end{lemma}
\begin{proof}
	The existence of $\eta_{f,N}$ follows from Paley-Wiener-Schwartz theorem (see \cite[Thm.\ 7.3.1]{Hormander2003}). To show that $\eta_{f,N}$ is uniformly computable from $N$ and $f$, notice that, since $\fouriertransform{f^{(n)}}(x)=(ix)^n \fouriertransform{f}(x)$, for every exponent $n$ we have
	\[ |\fouriertransform{f}(x)||x|^n = |\fouriertransform{f^{(n)}}(x)| \le \norm{f^{(n)}}{L^1}, \]
	and the latter is computable from a $\cont[0][c]{}(\mathbb{R},\mathbb{R})$-name of $f^{(n)}$. Using the binomial theorem, we can expand $|\fouriertransform{f}(x)|(1+|x|)^N$ into a finite sum, and then compute a bound for each of the terms.
\end{proof}

\begin{lemma}
	\thlabel{thm:bound_sum_psi}
	Let $\psi$ be a smooth function on $\torus$. For every $N>1$ there are $B_{\psi,N}>0$ and $M_0\in\mathbb{N}$ s.t.\ for every $M\ge M_0$
	\[ \sum_{m\in\mathbb{Z}\st |m|\ge M} |\fouriertransform{\psi}(m)| \le \frac{B_{\psi,N}}{M^{N-1}}. \]
	Moreover, the constants $B_{\psi,N}$ and $M_0$ are uniformly computable from $\psi$ and $N$.
\end{lemma}
\begin{proof}
	Using \thref{thm:schwartz_decay} and the integral test for the convergence of series we can write
	\begin{align*}
		\sum_{m\in\mathbb{Z}\st |m|\ge M} |\fouriertransform{\psi}(m)| & \le \eta_{\psi,N} \sum_{m\in\mathbb{Z}\st |m|\ge M} \frac{1}{(1+|m|)^N} \\
			& \le 2 \eta_{\psi,N} \left( \frac{1}{(1+M)^N} + \int_M^\infty \frac{1}{(1+x)^N}\dd{x} \right) \\
			& \le 2 \eta_{\psi,N} \left( \frac{1}{(1+M)^N} + \frac{1}{N-1}\frac{1}{(1+M)^{N-1}} \right).
	\end{align*}	
	We can then compute $M_0$ s.t.\ for every $M\ge M_0$
	\[ \frac{1}{(1+M)^N} \le \frac{1}{N-1}\frac{1}{(1+M)^{N-1}}  \]
	In particular, for every $M\ge M_0$
	\[ \sum_{m\in\mathbb{Z}\st |m|\ge M} |\fouriertransform{\psi}(m)| \le 2 \eta_{\psi,N}\frac{2}{N-1}\frac{1}{(1+M)^{N-1}} \le \frac{B_{\psi,N}}{M^{N-1}}, \]
	where $B_{\psi,N}:= 4 \eta_{\psi,N}$.
\end{proof}

From now on, we fix a computable non-negative smooth function $\phi$ with support in $[0,1]$ and $\int \phi(x)\dd{x} = 1$. For every $\zeta>0$ we define $\phi^\zeta(x):=\zeta^{-1}\phi(x/\zeta)$. Clearly the map $\function*{\mathbb{R}}{\cont[\infty][c]{}(\mathbb{R},\mathbb{R})}:=\zeta \mapsto \phi^\zeta$ is computable. Let also $\Phi^\zeta(x):=\sum_{k\in\mathbb{Z}} \phi^\zeta(x-k)$ be the periodization of $\phi^\zeta$ and $\Phi^\zeta_p(x):=\Phi^\zeta(px)$. Both $\Phi^\zeta$ and $\Phi^\zeta_p$ are uniformly computable in $\zeta$ and $p$.

For $M>2$, let $\mathbf{P}_M:= \{ p\in \mathbb{N} \st p$ is prime and $M/2<p\le M \}$. We define
\[ F^\zeta_M(x) := \frac{1}{|\mathbf{P}_M|} \sum_{p\in \mathbf{P}_M} \Phi^\zeta_p(x). \]
The map $F^\zeta_M$ is smooth, $1$-periodic, and $\int_{0}^{1} F^\zeta_M(x)\dd{x}=1$. In particular, we see it as a function in $L^1(\torus)$. Notice that the functions $F_M^\zeta$ are uniformly computable in $M$, and $\zeta$. 

\begin{lemma}
	\thlabel{thm:coefficients_FM}
	For every $\zeta$ and $M$ as above we have:
	\begin{enumerate}
		\item $\fouriertransform{F^\zeta_M}(0)=1$,
		\item if $0<|k| \le  M/2$ then $\fouriertransform{F^\zeta_M}(k)=0$,
		\item for every $N$ there is $C_N>0$, independent of $M$ and $\zeta$ and uniformly computable in $N$, s.t.\ for every $k\neq 0$,
		\[ |\fouriertransform{F^\zeta_M}(k)| \le C_N \frac{\log(|k|)}{M}\left(1+ \zeta\frac{|k|}{M}\right)^{-N}.\]
	\end{enumerate}
\end{lemma}
\begin{proof}
	Everything but the uniform computability of $C_N$ is proved in \cite[Ch.\ 9, pp.\ 69--70]{Wolff03}. In particular, the points $(1)$ and $(2)$ follow from the fact that $\fouriertransform{\Phi^\zeta_p}(k)$ is $\fouriertransform{\phi}(\zeta k/p)$ if $p|k$, and $0$ otherwise.
	Notice that, by the decomposition in prime factors, $|k|$ has at most $\log(|k|)/\log(M/2)$ divisors in $\mathbf{P}_M$. In particular, there is a computable constant $C$ independent of $|k|$ and $M$ s.t.\ $|k|$ has at most $C \log(|k|)/\log(M)$ divisors in $\mathbf{P}_M$.
	
	By the prime number theorem (see \cite[Sec.\ 22.19 and eq.\ (22.19.3)]{HardyWright08}), $|\mathbf{P}_M|$ is asymptotically distributed as $M/(2\log(M))$. In particular, there is a constant $A>0$ s.t.\ $|\mathbf{P}_M|\le A M/\log(M)$. The argument in \cite{Wolff03} shows that $C_N:=  A^{-1} C \eta_{\phi,N}$ satisfies the statement, where $\eta_{\phi,N}$ comes from \thref{thm:schwartz_decay}.
\end{proof}

From now on, we let $F_M:=F_M^{M^{-1-\alpha}}$. This guarantees that $\support(F_M)\subset \bigcup_{p\in \mathbf{P}_M} D_p$. Moreover, choosing $N=1$, the previous lemma states the existence of a constant $C_1$ s.t.\ for every $M>2$ and $k\neq 0$
\[ |\fouriertransform{F_M}(k)| \le C_1 M^{1+\alpha} \frac{\log(|k|)}{|k|}.\]

The following two lemmas provide the main technical tools to prove the effectiveness of the map $\alpha\mapsto S(\alpha)$, which will be proved in \thref{thm:s_alpha_effective}.

\begin{lemma}
	\thlabel{thm:estimate_FM}
	Let $\psi$ be a smooth non-negative function on $\torus$. There exists $C>0$ and $\tilde M\in \mathbb{N}$, uniformly computable in $\psi$ and $\alpha$, s.t.\ for every $M\ge \tilde M$ we have
	\begin{enumerate}
		\item for every $k\in\mathbb{Z}$, $|\fouriertransform{\psi F_M}(k)-\fouriertransform{\psi}(k)| \le C M^{-1} \log(M)$;
		\item for every $k\in\mathbb{Z}$ with $|k|>2M^{2+\alpha}$, $|\fouriertransform{\psi F_M}(k)-\fouriertransform{\psi}(k)| \le C M^{-1} \log(|k|)\left(1+ \frac{|k|}{M^{2+\alpha}}\right)^{-2}$.
	\end{enumerate}
\end{lemma}
\begin{proof}
	For the sake of readability, let us define, for $t\ge 1$, and $M>0$,
	\[ f_M(t):=\frac{\log(t)}{M}\left(1+ \frac{t}{M^{2+\alpha}}\right)^{-2} = \frac{\log(t)}{M}\left(\frac{M^{2+\alpha}}{M^{2+\alpha}+t}\right)^2. \]
	Notice that $f_M$ is strictly decreasing when $2t\log(t) \ge M^{2+\alpha}+t$, hence in particular for $t\ge M^{2+\alpha}$. Clearly, $f_M$ is uniformly computable in $\alpha$ and $M$. Moreover, given $\alpha$, we can uniformly compute a constant $T_\alpha$ s.t.\ for every $M>2$ and $t$, $f_M(t)\le T_\alpha M^{-1}\log(M)$. Indeed, for $s\ge M^{-2-\alpha}$,
	\begin{align*}
		f_M(sM^{2+\alpha}) & =\frac{(2+\alpha)\log(M)+ \log(s)}{M}\left(\frac{1}{1+s}\right)^{2} \\
			& = \frac{\log(M)}{M} \left(2+\alpha+\frac{\log(s)}{\log(M)}\right)\left(\frac{1}{1+s}\right)^{2} \le \frac{\log(M)}{M}\left(2+\alpha + \max_{s>0} \frac{\log(s)}{(1+s)^2} \right).
	\end{align*}

	To prove the first part of the lemma, let us notice that, by the known properties of the Fourier transform and by \thref{thm:coefficients_FM}, we have
	\[ |\fouriertransform{\psi F_M}(k)-\fouriertransform{\psi}(k)| = \left| \sum_{m\in\mathbb{Z}}\fouriertransform{\psi}(m)\fouriertransform{F_M}(k-m)-\fouriertransform{\psi}(k) \right|\le   \sum_{m\in\mathbb{Z}\st |k-m|>M/2}|\fouriertransform{\psi}(m)||\fouriertransform{F_M}(k-m)|. \]	
	Using \thref{thm:schwartz_decay} and a simple argument on power series, it is easy to show that there is a constant $A_\psi$, uniformly computable from $\psi$, s.t.\ $\sum_{m\in\mathbb{Z}}|\fouriertransform{\psi}(m)| \le A_\psi$. In particular,
	\begin{align*}
		|\fouriertransform{\psi F_M}(k)-\fouriertransform{\psi}(k)| & \le A_\psi \max_{m\in\mathbb{Z}\st |k-m|>M/2 } |\fouriertransform{F_M}(k-m)| \\
			& \le A_\psi C_2  \max_{m\in\mathbb{Z}\st |k-m|>M/2 }  f_M(|k-m|) \le C' M^{-1}\log(M),
	\end{align*}
	where $C_2$ is the constant provided by \thref{thm:coefficients_FM} and $C':=A_\psi C_2 T_\alpha$. This proves the first part of the claim.
	
	\smallskip
	Assume now that $|k|> 2M^{2+\alpha}$. We can write
	\[ |\fouriertransform{\psi F_M}(k)-\fouriertransform{\psi}(k)| \le  \sum_{m\in\mathbb{Z}\st 0<|k-m|\le |k|/2 }|\fouriertransform{\psi}(m)||\fouriertransform{F_M}(k-m)|+ \sum_{m\in\mathbb{Z}\st |k-m|> |k|/2 }|\fouriertransform{\psi}(m)||\fouriertransform{F_M}(k-m)|. \]

	For the first sum, notice that $|k-m|\le |k|/2$ implies $|m|\ge |k|/2$. Using \thref{thm:bound_sum_psi}, we can compute $\tilde M$ and $B_{\psi,3}$ s.t.\ for every $|k|\ge 2\tilde M$
	\[ \sum_{m\in\mathbb{Z}\st |m|\ge |k|/2 } |\fouriertransform{\psi}(m)| \le \frac{B_{\psi, 3}}{|k|^2}. \]
	Using also the fact that $\fouriertransform{F_M}(x)\le \fouriertransform{F_M}(0)=1$, we have
	\[ \sum_{m\in\mathbb{Z}\st 0<|k-m|\le |k|/2 }|\fouriertransform{\psi}(m)||\fouriertransform{F_M}(k-m)| \le \frac{B_{\psi, 3}}{|k|^2} \le B_{\psi, 3} f_M(|k|),  \]
	where the second inequality follows from the fact that, for $M>1$ and $|k|>2M^{2+\alpha}$, $|k|^{-2} \le f_M(|k|)$.

	The second sum can be majorized as follows:
	\begin{align*}
		\sum_{m\in\mathbb{Z}\st |k-m|> |k|/2 }|\fouriertransform{\psi}(m)||\fouriertransform{F_M}(k-m)| & \le A_\psi \max_{m\in\mathbb{Z}\st |k-m|> |k|/2 } C_2 f_M(|k-m|) \\
		& \le A_\psi C_2 f_M\left( \frac{|k|}{2} \right) \le 4 A_\psi C_2 f_M(|k|),
	\end{align*}
	where the second inequality follows from the fact that $f_M$ is decreasing for $|k-m|>|k|/2>M^{2+\alpha}$, while the last inequality follows from $f_M(|k|/2)\le 4 f_M(|k|)$.

	We can combine the two estimates to conclude that, for $|k|>2M^{2+\alpha}$ and $M>\tilde M$
	\[ |\fouriertransform{\psi F_M}(k)-\fouriertransform{\psi}(k)| \le C'' f_M(|k|), \]
	with $C'':=4A_\psi C_2 + B_{\psi,3}$.

	To conclude the proof it is enough to define $C:=\max\{C',C''\}$.
\end{proof}

Let us define, for $x\ge 0$
\[ g(x):= \begin{cases}
	x^{-\frac{1}{2+\alpha}}\log(x) & \text{if } x\ge x_0, \\
	x_0^{-\frac{1}{2+\alpha}}\log(x_0) 	& \text{otherwise},
\end{cases} \]
where $x_0:=e^{2+\alpha}$ is the maximum point of $g$. Notice that $g(x)$ is strictly decreasing for $x>x_0$.

\begin{lemma}
	\thlabel{thm:effective_G}
	Let $\psi$ be a non-negative smooth function on $\torus$. For every $\varepsilon>0$ and $M_0\in \mathbb{N}$ with $M_0>x_0$, there is a finite sequence $M_1<\hdots<M_N$, uniformly computable in $\alpha$, $\psi$, $\varepsilon$, and $M_0$, s.t.\ $M_0<M_1$ and, for every $k$,
	\[ |\fouriertransform{\psi G}(k)-\fouriertransform{\psi}(k)| \le \varepsilon g(|k|) \]
	where $G:=N^{-1}\sum_{i=1}^N F_{M_i}$.	
\end{lemma}
\begin{proof}
	Let $C,\tilde M$ be the constants provided by \thref{thm:estimate_FM}. We choose $N$ sufficiently large so that
	\[ \frac{C}{N}< \frac{\varepsilon}{4}. \]
	We also choose $M'$ sufficiently large so that $M'\ge \max\{ M_0, \tilde M \}$ and, for every $x\ge M'$,
	\[ C \frac{\log(x)}{x} \le \frac{\varepsilon}{4} g(x). \]
	The existence of such $M'$ follows from the fact that $x^{-1}\log(x) = o(g(x))$.

	Let us define $E_i(k):=|\fouriertransform{\psi F_{M_i}}(k)-\fouriertransform{\psi}(k)|$. We define the sequence $M_1<\hdots<M_N$ iteratively so that for every $j<N$
	\begin{itemize}
		\item $M_{j+1}> 2M_j^{2+\alpha}$,
		\item for every $|k| > M_{j+1}$, $\frac{1}{N} \sum_{i=1}^j E_i(k)\le \frac{\varepsilon}{4} g(|k|)$.
	\end{itemize}
	The second condition can always be satisfied as, by \thref{thm:estimate_FM},
	\begin{enumerate}
		\item for every $k$, $E_i(k) \le C M_i^{-1} \log(M_i)$;
		\item for every $k\in\mathbb{Z}$ with $|k|>2M_i^{2+\alpha}$, $E_i(k) \le C g(|k|)$.
	\end{enumerate}

	In fact, it is straightforward to see that, uniformly computably in $\psi$, $\alpha$, and $\varepsilon$, we can choose $N$, $M'$ and $\sequence{M_i}{1\le i \le N}$ so that they satisfy the above conditions.

	To show that the claim is satisfied with this choice of $N,M_1,\hdots,M_N$ we proceed as in \cite[Ch.\ 9, pp. 71--72]{Wolff03}. Let $j\in \{1,\hdots, N\}$ and let $k$ s.t.\ $M_j < |k| \le M_{j+1}$ (the cases $k\le M_1$ and $k> M_N$ are analogous).
	\begin{align*}
		|\fouriertransform{\psi G}(k)-\fouriertransform{\psi}(k)| & \le \frac{1}{N}\sum_{i=1}^N E_i(k)\\
			& = \frac{1}{N}\sum_{i=1}^{j-1} E_i(k) + \frac{1}{N} E_j(k) + \frac{1}{N}\sum_{i=j+1}^{N} E_i(k)  \\
			& \le \frac{\varepsilon}{4} g(|k|)+\frac{1}{N}\left(C\frac{\log(M_j)}{M_j}+Cg(|k|)\right)+ \frac{1}{N}\sum_{i=j+1}^{N} C \frac{\log(M_i)}{M_i}\\
			& \le  \frac{\varepsilon}{4} g(|k|)+ \frac{\varepsilon}{4} g(|k|)+ \frac{\varepsilon}{4} g(|k|)+ \frac{N-j}{N}C \frac{\log(M_{j+1})}{M_{j+1}} \le \varepsilon g(|k|).\qedhere
	\end{align*}
\end{proof}

We are finally ready to prove the effectiveness of the map $\alpha \mapsto S(\alpha)$.

\begin{theorem}
	\thlabel{thm:s_alpha_effective}
	The following maps are computable:
	\begin{gather*}
		\function*{\mathbb{R}\times \mathbb{N}}{\cantor}:= (\alpha,k)\mapsto P(\alpha,k)\\
		\function*{\mathbb{R}\times\mathbb{N}}{\hypCompactV([0,1])}:= (\alpha,n)\mapsto D_n(\alpha)\\
		\function*{\mathbb{R}\times \mathbb{N}}{\hypCompactV([0,1])}:= (\alpha,k)\mapsto \bigcup_{n\in P(\alpha,k)} D_n(\alpha)\\
		\function*{\mathbb{R}}{\hypCompactUF([0,1])}:= \alpha \mapsto S(\alpha)
	\end{gather*}
\end{theorem}
\begin{proof}
	Let us prove the computability of the first map. \thref{thm:effective_G} states the existence of a computable map $\Theta\pfunction{\mathbb{R}\times \cont[\infty][]{\torus}\times \mathbb{R}\times \mathbb{N}}{\baire}$ that maps $(\alpha, \psi, \varepsilon, M_0)$ to $\str{M_1,\hdots, M_N}$. 
	
	We recursively define two sequences $\sequence{\sequence{M_{i,m}}{i=1}^{N_m}}{m\ge 1}$ of finite sequences of natural numbers and $\sequence{G_m}{m\in\mathbb{N}}$ of smooth functions on $\torus$ as follows. We start letting $G_0$ be constantly equal to $1$. We then define
	\begin{gather*}
		\str{M_{1,m+1},\hdots, M_{N_{m+1},m+1}} := \Theta\left(\alpha, \prod_{i\le m}G_i, 2^{-m-2}, 10x_0+m\right)\\
		G_{m+1}:= \frac{1}{N_{m+1}}\sum_{i=1}^{N_{m+1}} F_{M_{i,m+1}}
	\end{gather*}
	where $x_0$ is defined as above.
		
	By \cite[Ch.\ 9, p.\ 73]{Wolff03}, a measure $\mu$ witnessing that $E(\alpha)$ is Salem is the weak-* limit of the sequence $\sequence{\mu_k}{k\in\mathbb{N}}$, where each $\mu_k$ is absolutely continuous w.r.t.\ the Lebesgue measure with density $\prod_{m\le k} G_m$. In particular, $\support(\mu_0)=[0,1]$ and $\support(\mu_{k+1})\subset \bigcup_{m=1}^{k+1} \support(G_m)=\bigcup_{m=1}^{k+1} \bigcup_{i=1}^{N_{m}} \support(F_{M_{i,m}})$. Since $\support(F_M)\subset \bigcup_{p\in \mathbf{P}_M} D_p$, we can define $P(\alpha,0):=\{1\}$ and
	\[ P(\alpha, k+1):= \bigcup_{m=1}^{k+1} \bigcup_{i=1}^{N_{m}} \mathbf{P}_{M_{i,m}}. \]
	The computability of this map follows from the computability of $\Theta$.
	\smallskip

	The computability of the other maps is then straightforward. A $\psi$-name (i.e.\ a full information name) for $D_n(\alpha)$ can be uniformly computed from $\alpha$ and $n$ as
	\[ D_n(\alpha)=\{ x\in [0,1] \st \min_{m\in\mathbb{Z}} | nx - m | \le n^{-1-\alpha} \} = \bigcup_{m\le n} \cball{\frac{m}{n}}{n^{-2-\alpha}} \cap [0,1]. \]	
	The computability of the third map follows from the computability of the first two, while the computability of the last map follows from the fact that $\bigcap\function{\infStrings{(\hypCompactUF([0,1]))}}{\hypCompactUF([0,1])}$ is computable (see \thref{thm:computability_union_intersection}).
\end{proof}

In particular, if $\alpha$ is computable then $S(\alpha)\in \lightfacePi^0_1(\mathbb{R})$. Notice however that, in the previous proposition, we only get a $\closedNegRep$-name (i.e.\ a negative representation name) for $S(\alpha)$. Indeed, the map $\bigcap\function{\infStrings{(\hypCompactV([0,1]))}}{\hypCompactV([0,1])}$ is not computable (it is not even continuous).

\section{The effective complexity of closed Salem sets}
\label{sec:eff_complexity_salem}

In this section we characterize the effective complexity of the conditions $\hdim(A)>p$, $\hdim(A)\ge p$, $\fourierdim(A)>p$, $\fourierdim(A)\ge p$ and ``$A$ is Salem'', i.e.\ we state and prove the effective counterparts of the results presented in \cite{MVSalem}.

We start by establishing the upper bounds for the complexity of the sets we are studying. Notice that, since $\lightfaceSigma^0_k(\hypCompactUF(X))\subset \lightfaceSigma^0_k(\hypCompactV(X))$ and $\lightfaceSigma^0_k(\hypClosedUF(X))\subset \lightfaceSigma^0_k(\hypClosedF(X))$, proving the upper bounds using $\hypCompactUF(X)$ or $\hypClosedUF(X)$ yield a stronger result.

\begin{proposition}
	\thlabel{thm:eff_hausdorff_complexity}
	For every $d$ and every compact $K\subset \mathbb{R}^d$,
	\begin{itemize}
		\item $\{ (A,p) \in\hypCompactUF(K)\times [0,d] \st \hdim(A)> p\}$ is $\lightfaceSigma^{0,K}_2$;
		\item $\{ (A,p) \in\hypCompactUF(K)\times [0,d] \st \hdim(A)\ge p\}$ is $\lightfacePi^{0,K}_3$.
	\end{itemize}
\end{proposition}
\begin{proof}
	For $A\in\hypCompactUF(K)$ let us define
	\[ D(A):= \{ s\in [0,d] \st (\exists \mu \in \ProbabilityMeas(A))(\exists c>0)(\forall x\in \mathbb{R}^d)(\forall r>0)(\mu(\ball{x}{r})\le c r^s ) \}. \]	
	Notice that, if $a\in \leftReal$ and $b\in \mathbb{R}$ (with the standard Cauchy representation) then the condition $a\le b$ is a $\lightfacePi^0_1$ predicate of $a$ and $b$ (as it is equivalent to $(\forall i)(p_a(i)\le b)$, where $p_a\in \repmap{\leftReal}^{-1}(a)$). Notice also that the map $(x,r)\mapsto \ball{x}{r}$ is computable. By \thref{thm:computability_measures}(1), $\mu(\ball{x}{r})\le c r^s$ is $\lightfacePi^0_1$ as a predicate of $\mu, x, r, c,$ and $s$.
	
	Moreover, $D(A)$ can be equivalently written as
	\[ \{ s\in [0,d] \st (\exists \mu \in \ProbabilityMeas(A))(\exists c>0)(\forall q_0\in \mathbb{Q}^d)(\forall q_1 \in \mathbb{Q}^+)(\mu(\ball{q_0}{q_1})\le c q_1^s ) \}. \]
	Indeed, one inclusion is obvious, while the other follows from the fact that for every $\varepsilon >0$ there are $q_0\in\mathbb{Q}^d$ and $q_1\in \mathbb{Q}$ s.t.\ $\ball{x}{r}\subset \ball{q_0}{q_1}$ and $q_1 < r +\varepsilon$. Hence
	\begin{align*}
		\mu(\ball{x}{r}) & \le \inf \{ \mu(\ball{q_0}{q_1}) \st \ball{x}{r}\subset \ball{q_0}{q_1} \text{ and } q_1 < r +\varepsilon \text{ and } \varepsilon >0 \} \\
		& \le \inf \{ cq_1^s \st \ball{x}{r}\subset \ball{q_0}{q_1} \text{ and } q_1 < r +\varepsilon \text{ and } \varepsilon >0  \}\\
		& \le \inf \{ c(r+\varepsilon)^s \st \varepsilon >0   \} = cr^s.
	\end{align*}
	Since the existential quantification on $c$ can be trivially restricted to the rationals, we have
	\[ S:= \{ (s,\mu)\in [0,d]\times \ProbabilityMeas(A)\st (\exists c>0)(\forall x\in \mathbb{R}^d)(\forall r>0)(\mu(\ball{x}{r})\le c r^s ) \}\in \lightfaceSigma^{0,A}_2~.\]
	Observe that $\mu\in \ProbabilityMeas(A)$ iff $\mu \in \ProbabilityMeas(K)$ and $\mu(A)\ge 1$. In particular, since $\hypCompactUF(K)$ is admissibly represented with the negative information representation, by \thref{thm:computability_measures}(2), given two names for $\mu$ and $A$, we can computably obtain a right-cut representation for $\mu(A)$, hence the the condition $\mu(A)\ge 1$ is a $\lightfacePi^{0}_1$ predicate of $\mu$ and $A$ (as if $x\in \rightReal$ the condition $x\ge 1$ is co-c.e.).
	Since $\ProbabilityMeas(K)$ is computably compact (\thref{thm:P(K)_comp_compact}), using (the relativized version of) \thref{thm:effective_AM}, we have
	\[ D(A)= \proj_{[0,d]} \{ (s,\mu) \in [0,d]\times \ProbabilityMeas(K) \st \mu(A)\ge 1 \land (s,\mu) \in S \} \in \lightfaceSigma^{0,A,K}_2~. \]
	To conclude the proof we notice that the conditions
	\begin{gather*}
		\hdim(A) > p \iff (\exists s\in \mathbb{Q})(s>p \land s\in D(A)),	\\
		\hdim(A) \ge p \iff (\forall s\in \mathbb{Q})(s<p \rightarrow s\in D(A))
	\end{gather*}
	are $\lightfaceSigma^{0,K}_2$ and $\lightfacePi^{0,K}_3$ respectively (as predicates of $A$ and $p$).
\end{proof}

\begin{proposition}
	\thlabel{thm:eff_fourier_complexity}
	For every $d$ and every compact $K\subset \mathbb{R}^d$,
	\begin{itemize}
		\item $\{ (A,p) \in\hypCompactUF(K)\times [0,d] \st \fourierdim(A)> p\}$ is $\lightfaceSigma^{0,K}_2$;
		\item $\{ (A,p) \in\hypCompactUF(K)\times [0,d] \st \fourierdim(A)\ge p\}$ is $\lightfacePi^{0,K}_3$.
	\end{itemize}
\end{proposition}
\begin{proof}
	Consider the set
	\[ D(A):=\{s\in[0,d]\st (\exists \mu \in \ProbabilityMeas(A))(\exists c>0)(\forall x\in \mathbb{R}^d)(|\fouriertransform{\mu}(x)|\le c|x|^{-s/2}) \}.\]
	Recall that, by definition,
	\[ \fouriertransform{\mu}(x) = \int e^{-i \scalarprod{x}{t}} \dd{\mu}(t) = \int \cos(\scalarprod{x}{t}) \dd{\mu}(t)-i\int \sin(\scalarprod{x}{t})\dd{\mu}(t)~. \]
	Since $\cos$ and $\sin$ are effectively bounded, by \thref{thm:computability_measures}.6 the map
	\[ \function*{\ProbabilityMeas(\mathbb{R}^d)\times \mathbb{R}}{\mathbb{R}}:=(\mu,x)\mapsto |\fouriertransform{\mu}(x)|\]
	is computable. By the continuity of the Fourier transform, the universal quantification on $x\in\mathbb{R}^d$ can be restricted to $\mathbb{Q}^d$. Since the quantification on $c$ can be trivially restricted to the rationals, we obtain that $D(A)= \proj_{[0,d]} Q$, with
\[
Q:= \{ (\mu,s)\in \ProbabilityMeas(K)\times [0,d] \st (\exists c\in \mathbb{Q}^+) (\forall q\in \mathbb{Q}^d)(\mu\in \ProbabilityMeas(A)\land|\fouriertransform{\mu}(q)|\le c|q|^{-s/2}) \}.
\]
The claim follows as in the proof of \thref{thm:eff_hausdorff_complexity}:
since the condition $\mu\in \ProbabilityMeas(A)$ is a $\lightfacePi^{0,K}_1$
predicate of $\mu$ and $A$ and $\ProbabilityMeas(K)$ is computably compact,
we have that $Q\in\lightfaceSigma^{0,A,K}_2$. Using (the relativized version
of) \thref{thm:effective_AM} we conclude that $D(A)\in
\lightfaceSigma^{0,A,K}_2$, and finally
	\begin{gather*}
		\fourierdim(A) > p \iff (\exists s\in \mathbb{Q})(s>p \land s\in D(A)),	\\
		\fourierdim(A) \ge p \iff (\forall s\in \mathbb{Q})(s<p \rightarrow s\in D(A))
	\end{gather*}
	are, respectively, a $\lightfaceSigma^{0,K}_2$ and a $\lightfacePi^{0,K}_3$ predicate of $A$ and $p$.
\end{proof}

\begin{corollary}
	\thlabel{thm:salem_eff_pi03}
	For every compact $K\subset\mathbb{R}^d$, the set $\{ A \in\hypCompactUF(K)\st A \in \Salem([0,d])\}$ is $\lightfacePi^{0,K}_3$.
\end{corollary}
\begin{proof}
	As in the proof of \cite[Thm.\ 3.4]{MVSalem}, recall that $\fourierdim(A)\le \hdim(A)$ for every Borel $A$, hence $\hdim(A)=\fourierdim(A)$ iff
	\[ (\forall r\in \mathbb{Q})( \hdim(A)>r \rightarrow \fourierdim(A)>r ), \]
	which is a $\lightfacePi^{0,K}_3$ condition by \thref{thm:eff_hausdorff_complexity} and \thref{thm:eff_fourier_complexity}.
\end{proof}

We now show that, if we take $d=1$ and $K=[0,1]$ then the above conditions are lightface complete for their respective classes. To do so, we use \thref{thm:s_alpha_effective} to prove an effective analogue of \cite[Lem.\ 3.6]{MVSalem}. We split the result in two lemmas.

Recall that, for $\alpha\ge 0$, $E(\alpha)$ is the set of $\alpha$-well approximable numbers, and $S(\alpha)$ is a closed Salem subset of $E(\alpha)$ with $\dim(E(\alpha))=\dim(S(\alpha))=2/(2+\alpha)$ (see Section~\ref{sec:eff_kaufman}).
\begin{lemma}
	\thlabel{thm:t_alpha_construction}
	For every rational $\alpha\ge 0$ there is a superset $T(\alpha)$ of $S(\alpha)$ with
	\[ T(\alpha) = \bigcap_{k\in\mathbb{N}} T^{(k)}(\alpha) =  \bigcap_{k\in\mathbb{N}} \bigcup_{j< N_k} J_j(\alpha,k), \]
	where the $J_j(\alpha,k)$ are pairwise disjoint (possibly degenerate) closed intervals, s.t.:
	\begin{itemize}
		\item $T(\alpha)$ is a closed Salem subset of $[0,1]$ with $\dim(T(\alpha))=\dim(S(\alpha))=2/(2+\alpha)$,
		\item the levels $T^{(k)}(\alpha)$ of the construction are s.t.
		\begin{enumerate}
			\item for every $k$, $T^{(k)}(\alpha)=\bigcup_{j< N_k} J_j(\alpha,k)$ and $T^{(k+1)}(\alpha) \subset T^{(k)}(\alpha)$;
			\item for every $k$ and every $i< N_k$ there exists $j< N_{k+1}$ s.t.\ $J_j(\alpha,k+1)\subset J_i(\alpha,k)$,
		\end{enumerate}
		\item the map $\function*{\mathbb{Q}\times \mathbb{N}}{\hypCompactV([0,1])}:=(\alpha,k)\mapsto T^{(k)}(\alpha)$ is computable.	
	\end{itemize}	

\end{lemma}
\begin{proof}
	We let $S^{(k)}(\alpha):=\bigcup_{n\in P(\alpha,k)} D_n(\alpha)$, so that $S(\alpha) = \bigcap_{k\in\mathbb{N}} S^{(k)}(\alpha)$. The set $S^{(k)}(\alpha)$ can be rewritten as
	\[ S^{(k)}(\alpha) = \bigcup_{i< M_k} I_i(\alpha,k) \]
	where, for each $k$, the intervals $I_i(\alpha,k)$ are closed, non-degenerate, and pairwise disjoint.

	We define $T^{(k)}(\alpha)$ recursively on $k$ as follows: at stage $0$ we let $T^{(0)}(\alpha):=S^{(0)}(\alpha)$. At stage $k+1$, let $\sequence{\tilde J_j(\alpha,k)}{j< \tilde N_{k+1} }$ be a finite sequence of closed mutually disjoint intervals s.t.\
	\[ \bigcup_{j<\tilde N_{k+1}} \tilde J_j(\alpha,k) = T^{(k)}(\alpha) \cap S^{(k+1)}(\alpha).\tag{$\star$}  \]
	For the sake of readability, let $\tilde T^{(k+1)}(\alpha):= \bigcup_{j<\tilde N_{k+1}} \tilde J_j(\alpha,k)$. Let also $W_k:=\{ i < N_k \st \tilde T^{(k+1)}(\alpha) \cap J_i(\alpha,k) = \emptyset \}$. We define
	\[ T^{(k+1)}(\alpha):=  \tilde T^{(k+1)}(\alpha) \cup \bigcup_{i\in W_k} \{ a_i \}, \]
	where $a_i$ is the left endpoint of $J_i(\alpha,k)$. We then have $N_{k+1} = \tilde N_{k+1} + |W_k|$.

	Clearly each $T^{(k)}(\alpha)$ is a finite union of closed intervals, hence $T(\alpha)$ is closed. Moreover, for every stage $k$, $T^{(k)}(\alpha)\setminus S^{(k)}(\alpha)$ is finite, and therefore $T(\alpha)\setminus S(\alpha)$ is countable. This implies that $\hdim(T(\alpha))=\fourierdim(T(\alpha))=\dim(S(\alpha))=2/(2+\alpha)$. It is straightforward to see that the levels $T^{(k)}(\alpha)$ of the construction of $T(\alpha)$ satisfy the conditions $(1)$ and $(2)$ in the statement of the lemma.

	Let us now prove the computability of the map $(\alpha,k)\mapsto T^{(k)}(\alpha)$. We define
	\[ \mathcal{E}:=\{ s+r^{n/m} \in [0,1] \st s,r\in\mathbb{Q} \text{ and } n,m\in\mathbb{N} \text{ with }m\neq 0\}. \]
	The set $\mathcal{E}$ can be naturally represented via the map $\repmap{\mathcal{E}}\pfunction{\Baire}{\mathcal{E}}:=p\mapsto q_{p(0)}+(q_{p(1)})^{p(2)/p(3)}$, where $\sequence{q_i}{i\in\mathbb{N}}$ is the canonical enumeration of $\mathbb{Q}^+$.

We notice that $\le$ (and hence $=$) are decidable when restricted to
$\mathcal{E}\times\mathcal{E}$. This follows from the decidability of the
theory of real closed fields \cite{TarskiRCF}: notice that $\mathcal{E}$ is
definable in the first-order language
$\mathcal{L}_{\mathrm{rcf}}=(+,\cdot,\le,0,1)$ of real closed fields. Indeed,
the rational number $a/b$ can be defined as the unique $y$ that satisfies
$(1+\hdots+1)\cdot y = 1+\hdots+1$, where the first sum involves $b$ $1$s,
and the second one involves $a$ $1$s. Moreover, for every $s,r\in \mathbb{Q}$
and $n,m\in\mathbb{N}$, the formula $x=s+r^{n/m}$ can be written as
$(x-s)^m=r^n$, and therefore it is expressible in
$\mathcal{L}_{\mathrm{rcf}}$. In particular, this shows that $\le$ and $=$
are decidable for elements of $\mathcal{E}$.

	Since, by definition, 	
	\[ D_n(\alpha)= \bigcup_{m\le n} \cball{\frac{m}{n}}{n^{-2-\alpha}} \cap
[0,1], \] the endpoints of the intervals $I_i(\alpha,k)$ are of the form
$s+r^{2+\alpha}$, for some $s,r\in\mathbb{Q}$. In particular, if
$\alpha\in\mathbb{Q}$ then the endpoints of the $k$-th level intervals of
$S(\alpha)$ are in $\mathcal{E}$. The uniform computability of the finite set
$P(\alpha,k)$ in the definition of $S^{(k)}(\alpha)$
(\thref{thm:s_alpha_effective}) implies that, for each $k$, we can think of
$S^{(k)}(\alpha)$ as being represented via the sequence
$\sequence{(a_{i,k},b_{i,k})}{i< M_k}$ in $\mathcal{E}\times\mathcal{E}$,
where $I_i(\alpha,k)=[a_{i,k},b_{i,k}]$.

	We now show that, for each $k$, the endpoints of the $k$-th level intervals of $T(\alpha)$ are in $\mathcal{E}$, and that $T^{(k)}(\alpha)$ can be uniformly represented as the sequence in $\mathcal{E}\times \mathcal{E}$ of the endpoints of the intervals $\sequence{J_j(\alpha,k)}{j<N_k}$. We proceed by induction on $k$. At stage $0$ the statement is trivial. Assume the claim holds for $T^{(k)}(\alpha)$. The decidability of $\le\restrict{\mathcal{E}\times\mathcal{E}}$ implies that given two finite sequences $\sequence{U_n}{n}$ and $\sequence{V_m}{m}$ of intervals with endpoints in $\mathcal{E}$ (where each interval is represented via a pair of $\repmap{\mathcal{E}}$-names for its endpoints), we can uniformly compute a finite sequence $\sequence{W_\ell}{\ell}$ of mutually disjoint intervals with endpoints in $\mathcal{E}$ s.t.
	\[ \bigcup_\ell W_\ell = \bigcup_n U_n \cap \bigcup_m V_m . \]
	In particular, this implies that a sequence $\sequence{\tilde J_j(\alpha,k)}{j< \tilde N_{k+1} }$ that satisfies $(\star)$ can be uniformly computed from $\alpha$ and $k$.

	Similarly, for every $i<N_k$ and every $j< \tilde N_{k+1}$, we can uniformly (in $\alpha$ and $k$) decide whether $\tilde J_j(\alpha, k+1) \cap J_i(\alpha,k)=\emptyset$. In other words, the set $W_k$ is uniformly computable, and hence so is $|W_k|$.

	It is now straightforward to see that, given $\alpha$ and $k$, we can uniformly compute a finite sequence $\sequence{(c_{j,k},d_{i,k})}{j< N_{k+1}}$ in $\mathcal{E}\times\mathcal{E}$ s.t.\ the $k$-th level intervals of $T(\alpha)$ are $J_j(\alpha,k)=[c_{j,k},d_{j,k}]$. This, in turn, implies the computability of the map $(\alpha,k)\mapsto T^{(k)}(\alpha)$.
\end{proof}

\begin{lemma}
	\thlabel{thm:sigma02_map_salem_eff}
	There is a computable function $f\function{[0,1]_<\times \Cantor}{\hypCompactV([0,1])}$ s.t.\ for every $p,x$, $f(p,x)$ is Salem and
	\[ \dim(f(p,x)) = \begin{cases}
		p & \text{if } x\in Q_2\\
		0 & \text{if } x\notin Q_2
	\end{cases}\]
\end{lemma}
\begin{proof}
	The proof of the lemma follows a similar strategy as the proof of \cite[Lem.\ 3.6]{MVSalem}. Let $T(\alpha)$ and $T^{(k)}(\alpha)$ be as in \thref{thm:t_alpha_construction}.
	For every interval $I=[a,b]$ and every $k$ let $T^{(k)}(\alpha, I)$ be the set obtained by scaling $T^{(k)}(\alpha)$ to the interval $I$. Notice that the mapping $x\mapsto a+(b-a)x$ computably sends $[0,1]$ onto $I$, it is affine and it is invertible if $I$ is non-degenerate. In particular, the partial map
	\[ \function*{\mathbb{N}\times \mathbb{Q}\times \hypCompactV([0,1])}{\hypCompactV([0,1])}:=(k,\alpha,I)\mapsto T^{(k)}(\alpha,I)\]
	is computable.

We first define a map $g\pfunction{\mathbb{Q}\times
\Cantor}{\hypCompactV([0,1])}$ s.t.\ for every $q\in [0,1)$ and
$x\in\Cantor$, $g(q,x)$ is Salem and $\dim(g(q,x))=q$ if $x\in Q_2$ and $0$
otherwise. If $q=0$ we just take $g(q,x):=\emptyset$. Assume $q\in (0,1)\cap
\mathbb{Q}$ and let $\alpha=2(1-q)/q$ so that $2/(2+\alpha)=q$. We define
$F^{(k)}_x$ recursively as
	\begin{description}
		\item[Stage $k=0$]: $F^{(0)}_x := [0,1]$;
		\item[Stage $k+1$]: Let $J^{(k)}_0,\hdots,J^{(k)}_{M_k}$ be the disjoint closed intervals s.t.\ $F^{(k)}_x=\bigcup_{i\le M_k} J^{(k)}_i$.

If $x(k+1)=1$ then let $\rho \in \mathbb{Q}$ be such that $(2
\rho)^{2^{-k}} (M_k+1) \le 2^{-k}$. For each $i\le M_k$ let $J^{(k)}_i=[a_i,b_i]$ and define
\[ H^{(k)}_i:=\cball{\frac{a_i+b_i}{2}}{\min\left\{ \frac{b_i-a_i}{2}, \rho\right\}}.\]
The choice of $\rho$ implies that 		
\[ \sum_{i\le M_k} \diam{H^{(k)}_i}^{2^{-k}} \le 2^{-k}. \]
Define then $F^{(k+1)}_x:= \bigcup_{i\le M_k} H^{(k)}_i$.

		If $x(k+1)=0$ then let $s\le k$ be largest s.t.\ $x(s)=1$ (or $s=0$ if there is none). For each $i\le M_s$, apply the $(k+1-s)$-th step of the construction of $T(\alpha, J^{(s)}_i)$. Define $F^{(k+1)}_x:=\bigcup_{i\le M_s} T^{(k+1-s)}(\alpha, J^{(s)}_i)$.
	\end{description}
	We then define $g(q,x):= F_x := \bigcap_{k\in\mathbb{N}} F^{(k)}_x$. Clearly $F_x$ is closed, as intersection of closed sets.
	
	Let us show that it is Salem with the prescribed dimension. Assume first that $x\in Q_2$ (i.e.\ $x$ is eventually $0$) and let $s$ be the largest index s.t.\ $x(s)=1$ (or $s=0$ if there is none). By construction $F_x=\bigcup_{i\le M_s} T(\alpha,J^{(s)}_i)$. Since $\dim(T(\alpha,J^{(s)}_i))=q$ and each $T(\alpha,J^{(s)}_i)$ is closed, we have that $\dim(F_x)=\max\{\dim(T(\alpha,J^{(s)}_i)) \st i\le M_s\}=q$. On the other hand, if $x\notin Q_2$ then we show that for each $s>0$ and each $\varepsilon>0$ there is a cover $\sequence{A_n}{n\in\mathbb{N}}$ of $F_x$ s.t.\ $\sum_{n\in\mathbb{N}}\diam{A_n}^s \le \varepsilon$, which implies that $\hmeas^s(F_x)=0$. For every $s$ and $\varepsilon$ we can pick $k$ a sufficiently large $k$ so that $2^{-k}\le \min\{s,\varepsilon\}$ and $x(k+1)=1$. The intervals $\sequence{H^{(k)}_i}{i\le M_k}$ defined in the construction of $F_x$ form a cover of $F_x$ s.t.\
	\[ \sum_{i\le M_k} \diam{H^{(k)}_i}^s \le \sum_{i\le M_k} \diam{H^{(k)}_i}^{2^{-k}} \le 2^{-k} \le \varepsilon. \]
	
	We now show that $g$ is computable, i.e.\ that a full information name for $F_x$ can be uniformly computed from $q$ and $x$. Notice that, since the map $(k,\alpha,I)\mapsto T^{(k)}(\alpha,I)$ is computable, then so is the map $(k,p,x)\mapsto F^{(k)}_x$ (where the codomain is represented with the full information representation). Hence, a $\closedNegRep$-name for $F_x$ can be computed from a sequence $\sequence{r_k}{k\in\mathbb{N}}$ where $r_k$ is a $\closedNegRep$-name for $F^{(k)}_x$ (\thref{thm:computability_union_intersection}). To compute a $\psi_+$-name for $F_x$ (i.e.\ a positive information name), we use the fact that no interval is ever entirely removed and that no interval is entirely contained in $F_x$ (as $\dim(F_x)<1$). In particular, a $\psi_+$-name for $F_x$ is obtained by listing all the basic open balls $U$ s.t.\ there are $k$ and $i\le M_k$ s.t.\ $U$ contains a $k$-th level interval $J^{(k)}_i$. Notice that, since no interval is entirely removed, $ J^{(k)}_i\subset U$ implies $U\cap F_x\neq \emptyset$. Moreover, if $V\cap F_x\neq\emptyset$ for some basic open ball $V$, then for some $k$ and $i\le M_k$, $V$ contains the $k$-th level interval $J^{(k)}_i$.
	
We now define a map $f$ that satisfies the statement of the lemma. Let
$\sequence{I_n}{n\in\mathbb{N}}$ be the sequence of disjoint intervals $I_n:=
[2^{-2n-1}, 2^{-2n }]$ and $\sequence{\tau_n}{n\in\mathbb{N}}$ a uniformly
computable sequence of similarity transformations $\tau_n\function{[0,1]}{I_n}$. Any
$p\in [0,1]_<$ is given as a sequence $\sequence{q_n}{n\in\mathbb{N}}$ of
rationals in $[0,1]$ which is monotonically increasing and converges to $p$.
We define
\[
f(p,x):=  \{0\} \cup \bigcup_{n\in\mathbb{N}} \tau_n g(q_n,x).
\]
The fact that $f(p,x)$ is Salem and has the prescribed dimension follows from
the properties of $g$ and the countable stability for closed sets of $\hdim$
and $\fourierdim$. 	Notice that a $\repmap{\hypCompactV([0,1])}$-name for
$f(p,x)$ can be obtained by carefully merging the
$\repmap{\hypCompactV([0,1])}$-names of the sets $\tau_n g(q_n,x)$. We can
briefly sketch the argument as follows: a basic open set intersects $f(p,x)$
iff it intersects $\tau_n g(q_n,x)$ for some $n$. On the other hand, to list
the basic open balls contained in the complement of $f(p,x)$ it suffices to
list all the basic open balls contained in the relative topology of $I_n
\setdifference \tau_n g(q_n,x)$ together with the open intervals
$(2^{-2n-2},2^{-2n-1})$. The claim follows from the fact that the intervals
$I_n$ are uniformly co-c.e.\ closed.
\end{proof}

\medskip

The following results are the effective counterparts of \cite[Prop.\ 3.7, Thm.\ 3.8 and Thm.\ 3.9]{MVSalem}.

\begin{theorem}
	For every $p<1$ the sets
	\begin{align*}
		& \{ A \in\hypCompactV([0,1]) \st \hdim(A)> p\},\\
		& \{ A \in\hypCompactV([0,1]) \st \fourierdim(A)> p\}
	\end{align*}
	are $\lightfaceSigma^0_2$-complete.
\end{theorem}
\begin{proof}
	The upper bounds have been shown in \thref{thm:eff_hausdorff_complexity} and \thref{thm:eff_fourier_complexity}. The hardness is a corollary of \thref{thm:sigma02_map_salem_eff}.
\end{proof}
	
Recall that $P_3$ is the $\lightfacePi^0_3$-complete subset of $2^{\mathbb{N}\times\mathbb{N}}$ defined as
\[ P_3 := \{ x \in 2^{\mathbb{N}\times\mathbb{N}} \st (\forall m)(\forall^\infty n)( x(m,n) = 0 ) \}. \]
The proof of the following theorem is similar to the proof of \cite[Thm.\ 3.8]{MVSalem}, using \thref{thm:eff_hausdorff_complexity}, \thref{thm:eff_fourier_complexity} and \thref{thm:sigma02_map_salem_eff} in place of, respectively, \cite[Prop.\ 3.2, Prop.\ 3.3 and Lem.\ 3.6]{MVSalem}.

\begin{theorem}
	\thlabel{thm:dim_>=p_pi03_complete_eff}
	There exists a computable map $F\function{(0,1]_<\times 2^{\mathbb{N}\times\mathbb{N}}}{\hypCompactV([0,1])}$ s.t.\ for every $p$ and $x$, $F(p,x)$ is a Salem set and $\dim(F(p,x))\ge p$ iff $x\in P_3$. For every computable $q\in (0,1]$, letting
	\begin{align*}
		& X_1:= \{ A \in\hypCompactV([0,1]) \st \hdim(A)\ge q\},\\
		& X_2:= \{ A \in\hypCompactV([0,1]) \st \fourierdim(A)\ge q\}
	\end{align*}
	we have that every set $X$ s.t.\ $X_2\subset X \subset X_1$ is $\lightfacePi^0_3$-hard. In particular, $X_1$ and $X_2$ are $\lightfacePi^0_3$-complete.
\end{theorem}
\begin{proof}
	After showing the existence of a computable $F$ as claimed, the other statements follow. In particular, the completeness of $X_1$ and $X_2$ follows from \thref{thm:eff_hausdorff_complexity} and \thref{thm:eff_fourier_complexity}.
	
	For the first part, consider the computable map $\Phi\colon 2^{\mathbb{N}\times\mathbb{N}}\to 2^{\mathbb{N}\times\mathbb{N}}$ defined as $\Phi(x)(m,n):= \max_{i\le m} \, x(i,n)
	$ and notice that $x\in P_3$ iff $\Phi(x) \in P_3$. Intuitively, $\Phi(x)$ is a computable modification of $x$ s.t.\ the set of rows with finitely many $1$'s is an initial segment of $\mathbb{N}$ (it is $\mathbb{N}$ iff $x\in P_3$).

	For every $m$, let $I_m:=[2^{-2m-1},2^{-2m}]$ and $q_{m}:=p(1-2^{-m-1})$. Fix also a similarity transformation $\tau_m \function{[0,1]}{I_m}$ and define $g_m\function{\Cantor}{\hypCompactV(I_m)}$ as $g_m:= \tau_m \circ f(q_m, \cdot)$, where $f$ is the computable map provided by \thref{thm:sigma02_map_salem_eff}. In particular,
	\[ \dim(g_m(y))= \begin{cases}
		q_m & \text{if } y \in Q_2,\\
		0 	& \text{if } y \notin Q_2.
	\end{cases} \]

	We define
	\[ F(p,x):= \{0\} \cup \bigcup_{m\in\mathbb{N}} g_m(\Phi(x)_m), \]
	where $\Phi(x)_m$ denotes the $m$-th row of $\Phi(x)$. The accumulation point $0$ is added to ensure that $F(p,x)$ is a closed set.

	The computability of $F$ follows from the computability of $\Phi$ and $g$. Using the stability properties of the Hausdorff and Fourier dimensions, we have that $F(p,x)$ is Salem and
	\[ \dim(F(p,x))= \sup_{m\in\mathbb{N}} \hdim( g_m(\Phi(x)_m) ) = \sup_{m\in\mathbb{N}} \fourierdim( g_m(\Phi(x)_m) ).\]
	In particular, if $x\in P_3$ then $\Phi(x)\in P_3$ and, for every $m$, $\Phi(x)_m\in Q_2$, hence
	\[ \dim(F(p,x))= \sup_{m\in\mathbb{N}} \dim(g_m(\Phi(x)_m)) =\sup_{m\in\mathbb{N}} q_m  = p. \]
	On the other hand, if $x\notin P_3$ then there is a $k>0$ s.t.\ for every $m\ge k$, $\Phi(x)_m\notin Q_2$ and hence $\dim(g_m(\Phi(x)_m))=0$. This implies that
	\[ \dim(F(p,x))  \le q_k < p, \]
	and this completes the proof.
\end{proof}

\begin{theorem}
	\thlabel{thm:salem_pi03_complete}
	The set $\{A\in\hypCompactV([0,1]) \st A\in \Salem([0,1])\}$ is $\lightfacePi^0_3$-complete.
\end{theorem}
\begin{proof}
	The upper bound was proved in \thref{thm:salem_eff_pi03}. To prove the hardness, fix a computable $p>0$ and let $K\in \hypCompactV([0,1])$ be a computable set s.t.\ $\hdim(K)=p$ and $\fourierdim(K)=0$ (e.g.\ we can choose $K$ to be the Cantor middle-third set). Let also $F$ be the map provided by \thref{thm:dim_>=p_pi03_complete_eff} and define the map $h\function{2^{\mathbb{N}\times\mathbb{N}}}{\hypCompactV([0,1])}$ as
	\[ h(x):= F(p,x)\cup K.  \]
	The computability of $h$ follows from the computability of $p$ and $F$, and the fact that the union map $\cup\function{\hypCompactV([0,1])\times\hypCompactV([0,1])}{\hypCompactV([0,1])}$ is computable (see \thref{thm:computability_union_intersection}). Moreover
	\begin{gather*}
		\hdim(h(x)) = \max \{ \dim(F(p,x)), p \},\\
		\fourierdim(h(x)) = \dim(F(p,x)).
	\end{gather*}
	In particular, $h(x)$ is Salem iff $\dim(F(p,x))\ge p$ iff $x\in P_3$.
\end{proof}

\bigskip

We now turn our attention to the closed Salem subsets of $X$, where $X$ is $[0,1]^d$ or $\mathbb{R}^d$. We first notice the following result, which follows from the proofs of \thref{thm:eff_hausdorff_complexity} and \thref{thm:eff_fourier_complexity}.

\begin{lemma}
	\thlabel{thm:comp_compact_Rd}
	${}$
	\begin{itemize}
		\item $\{ (K,p) \in\hypCompactV(\mathbb{R}^d)\times [0,d] \st \hdim(K)> p\}$ is $\lightfaceSigma^0_2$;
		\item $\{ (K,p) \in\hypCompactV(\mathbb{R}^d)\times [0,d] \st \hdim(K)\ge p\}$ is $\lightfacePi^0_3$;
		\item $\{ (K,p) \in\hypCompactV(\mathbb{R}^d)\times [0,d] \st \fourierdim(K)> p\}$ is $\lightfaceSigma^0_2$;
		\item $\{ (K,p) \in\hypCompactV(\mathbb{R}^d)\times [0,d] \st \fourierdim(K)\ge p\}$ is $\lightfacePi^0_3$.		
	\end{itemize}
\end{lemma}
\begin{proof}
	We only prove the statement for $\hdim(K)>p$, the proof of the complexity of the other sets is analogous. Let
	\[ X_n := \{ (K,p) \in\hypCompactV([-n,n]^d)\times [0,d] \st \hdim(K)> p\}. \]
	Since $\hypCompactV([-n,n]^d)$ computably embeds into $\hypCompactV(\mathbb{R}^d)$, we can see $X_n$ as a subset of $\hypCompactV(\mathbb{R}^d)\times [0,d]$. In particular, for every $(K,p)\in \hypCompactV(\mathbb{R}^d)\times [0,d]$,
	\[ \hdim(K)>p \iff (\exists n)((K,p) \in X_n). \]
	Hence, it is enough to show that the sets $\sequence{X_n}{n\in\mathbb{N}}$ are uniformly $\lightfaceSigma^0_2$, i.e.\ that
	\[ \{ (n, K, p) \st (K,p) \in X_n \} \in \lightfaceSigma^0_2(\mathbb{N}\times \hypCompactV(\mathbb{R}^d)\times [0,d]). \]
	Notice that, since the sets $\sequence{[-n,n]^d}{n\in\mathbb{N}}$ are uniformly computably compact, then so are the sets $\sequence{\hypCompactV([-n,n]^d)}{n\in\mathbb{N}}$ (the argument of \thref{thm:K(X)_computably_compact} can be run uniformly in $n$). This, in turn, implies that the set
	\[ \{ (n,K,p)\st (K,p)\in  X_n \}\]
	is $\lightfaceSigma^0_2$, as the argument in the proof of \thref{thm:eff_hausdorff_complexity} can be run uniformly in $n$.	
\end{proof}

This result can be used to obtain the upper bounds in the non-compact case, i.e.\ the effective counterpart of the upper bounds obtained in \cite[Thm.\ 5.4 and Thm.\ 5.5]{MVSalem}.

\begin{proposition}
	\thlabel{thm:eff_complexity_haus_fourier_higher_dim}
	${}$
	\begin{itemize}
		\item $\{ (A,p) \in\hypClosedF(\mathbb{R}^d)\times [0,d] \st \hdim(A)> p\}$ is $\lightfaceSigma^0_2$;
		\item $\{ (A,p) \in\hypClosedF(\mathbb{R}^d)\times [0,d] \st \hdim(A)\ge p\}$ is $\lightfacePi^0_3$;
		\item $\{ (A,p) \in\hypClosedF(\mathbb{R}^d)\times [0,d] \st \fourierdim(A)> p\}$ is $\lightfaceSigma^0_2$;
		\item $\{ (A,p) \in\hypClosedF(\mathbb{R}^d)\times [0,d] \st \fourierdim(A)\ge p\}$ is $\lightfacePi^0_3$;
		\item $\{ A \in\hypClosedF(\mathbb{R}^d)\st A\in\Salem([0,d]) \}$ is $\lightfacePi^0_3$.
	\end{itemize}
\end{proposition}
\begin{proof}
	We only prove the statement for the Hausdorff dimension, the proof of the complexity of the Fourier dimension is analogous (as both are stable under countable union of closed sets), and the result on the complexity of the Salem sets can be obtained as in \thref{thm:salem_eff_pi03}.
	
	Notice that, since $A$ is closed,
	\[ \hdim(A)>p \iff (\exists K \in \hypClosedF(\mathbb{R}^d))(K\subset A \land K\in \hypCompactV(\mathbb{R}^d) \land \hdim(K)>p)~.\]
	Notice that, if $F$, $G$ are two closed sets represented with the full information representation, the predicate $F\subset G$ is $\lightfacePi^0_1$ as a predicate of $F$ and $G$. In fact we can prove something slightly stronger: if $p_F$ is a $\closedPosRep$-name (positive information name) for $F$ and $q_G$ is a $\closedNegRep$-name (negative information name) for $G$ then the condition $F\subset G$ is $\lightfacePi^0_1$ in $p_F$ and $q_G$. Indeed,
	\[ F\subset G \iff \setcomplement{G}\cap F =\emptyset \iff (\forall i)(\forall j)(q_G(i)\neq p_F(j)).\]
This shows that $K\subset A$ and $K\in \hypCompactV(\mathbb{R}^d)$ are respectively $\lightfacePi^0_1$ (as a predicate of $K$ and $A$) and an effective union of $\lightfacePi^0_1$ sets (as a predicate of $K$, as it is equivalent to $(\exists n)(K\subset [-n,n]^d)$). Moreover, since the
inclusion map $\hypClosedF(X)\restrict{\hypCompact(X)}\hookrightarrow \hypCompactV(X)$ is computable, using \thref{thm:comp_compact_Rd} we have that $\{(K,p) \in \hypClosedF(\mathbb{R}^d)\times [0,d] \st K\in \hypCompactV(\mathbb{R}^d) \land \hdim(K)>p\}$ is an effective union of $\lightfacePi^0_1$ sets. Thus also 
\[ \{ (K,A,p) \in \hypClosedF(\mathbb{R}^d)\times \hypClosedF(\mathbb{R}^d) \times [0,d] \st K\subset A \land K\in \hypCompactV(\mathbb{R}^d) \land \hdim(K)>p\}\]
is an effective union of $\lightfacePi^0_1$ sets. Since $\hypClosedF(\mathbb{R}^d)$ is computably compact
(\thref{thm:hypclosed_Rd_eff_compact}) we can apply \thref{thm:effective_AM_eff_spaces} and conclude that
	\[ \{ (A,p) \in\hypClosedF(\mathbb{R}^d)\times [0,d] \st \hdim(A)> p\}\text{ is } \lightfaceSigma^0_2~.\]
	Since $\hdim(A)\ge p$ iff $(\forall r\in \mathbb{Q})(r< p \rightarrow\hdim(A)>r)$, this also shows that $\hdim(A)\ge p$ is a $\lightfacePi^0_3$ predicate of $A$ and $p$.
\end{proof}

\smallskip

We now turn our attention to the lower bounds for the complexity of the above conditions. In \cite[Sec.\ 4]{MVSalem}, we exploited a recent construction of a higher dimensional analogue of $E(\alpha)$ (introduced in \cite{FraserHambrook2019}) to show that, for a closed $A\subset [0,1]^d$, the conditions $\hdim(A)>p$ and $\fourierdim(A)>p$ are $\boldfaceSigma^0_2$-complete (when $p<d$) and the conditions $\hdim(A)\ge q$ and $\fourierdim(A)\ge q$ are $\boldfacePi^0_3$-complete (when $q>0$). However, we are not aware of any proof of the effectiveness of the arguments presented in \cite{FraserHambrook2019}, which would be needed to prove a higher-dimensional analogue of \thref{thm:sigma02_map_salem_eff}.

However, we use a classical theorem of Gatesoupe to obtain a (slightly weaker) result, namely the completeness of the above conditions when $p$ and $q$ are sufficiently large.

\begin{theorem}
	\thlabel{thm:eff_completeness_n_dim}
	Let $X$ be $[0,1]^d$ or $\mathbb{R}^d$. For every $p\in [d-1,d)$ the sets
	\begin{align*}
		& \{ A \in\hypClosedF(X) \st \hdim(A)> p\},\\
		& \{ A \in\hypClosedF(X) \st \fourierdim(A)> p\}
	\end{align*}
	are $\lightfaceSigma^0_2$-complete. For every computable $q\in (d-1,d]$, the sets
	\begin{align*}
		& \{ A \in\hypClosedF(X) \st \hdim(A)\ge q\},\\
		& \{ A \in\hypClosedF(X) \st \fourierdim(A)\ge q\},\\
		& \{ A \in\hypClosedF(X)\st A \in \Salem(X)\}
	\end{align*}
	are $\lightfacePi^0_3$-complete.
\end{theorem}

\begin{proof}
	By \thref{thm:eff_complexity_haus_fourier_higher_dim}, it is enough to show that the above sets are hard for their respective class.
	
Recall that, by a theorem of Gatesoupe \cite{Gatesoupe67}, if $A\subset [0,1]$ has at least a point different from $0$ and is Salem with dimension $\alpha$ then the set $\tilde{A}:= \{x\in [0,1]^d \st |x|\in A\}$ is Salem with dimension $d-1+\alpha$. It is easy to see that the map $r\function{\hypCompactV([0,1])}{\hypCompactV([0,1]^d)}:=A\mapsto \tilde{A}$ is computable.
	
	To show that the first two sets are $\lightfaceSigma^0_2$-hard, let $f$ be the map provided by \thref{thm:sigma02_map_salem_eff}. For $x\in \Cantor$ and $p\in [d-1,d)$, we have
	\[ x\in Q_2 \iff \hdim(r(f(1,x)))>p \iff \fourierdim(r(f(1,x)))>p. \]
	
	Fix now a computable $q\in (d-1,d]$. To show that the conditions $\hdim(A)\ge q$ and $\fourierdim(A)\ge q$ are $\lightfacePi^0_3$-hard, consider a sequence $\sequence{C_n}{n\in\mathbb{N}}$ of mutually disjoint closed cubes s.t.\

	\begin{itemize}
		\item $C_n\subset [0,1]^d$,
		\item $\closure{\bigcup_{n\in\mathbb{N}} C_n} = \{\mathbf{0} \} \cup \bigcup_{n\in\mathbb{N}} C_n$, where $\mathbf{0}$ is the origin of the $d$-dimensional Euclidean space,
		\item the sets have uniformly computable $\psi$-names, i.e.\ there is a computable map that, given $n$, produces a $\psi$-name for $C_n$.
	\end{itemize}
	It is easy to provide examples of sequences of closed sets that satisfy the above conditions. In particular, the last point guarantees that the similarity transformations $\tau_n\function{[0,1]^d}{C_n}$ are uniformly computable. The claim follows by a straightforward adaptation of the proof of \thref{thm:dim_>=p_pi03_complete_eff}.
	
	Finally, to show that the family of closed Salem sets is $\lightfacePi^0_3$-hard, we adapt the proof of \thref{thm:salem_pi03_complete}, where the compact set $K$ is replaced with a fixed computable compact set $Y\subset C_0$ with null Fourier dimension and Hausdorff dimension $d$.
\end{proof}

\section{The Weihrauch degree of Hausdorff and Fourier dimension}
\label{sec:salem_weihrauch}

In this section, we briefly show how the results obtained in the previous sections can be used to characterize the uniform strength of the maps computing the Hausdorff and Fourier dimension of a closed subset of $\mathbb{R}^d$, for some fixed $d$.

Let $f\pmfunction{X}{Y}$ and $g\pmfunction{Z}{W}$ be partial multi-valued functions between represented spaces. We say that $f$ is \textdef{Weihrauch reducible} to $g$ ($f\weireducible g$) iff there are two computable maps $\Phi,\Psi\pfunction{\Baire}{\Baire}$ s.t., for every realizer $G$ of $g$, the map $p\mapsto \Psi(\coding{ p,G\Phi(p) })$ is a realizer for $f$. A thorough presentation on Weihrauch reducibility is out of the scope of this paper, and the reader is referred to \cite{BGP17}.

We define the \textdef{compositional product} as
\[ f\compproduct g:=\max_{\weireducible}\{f_0\circ g_0\st f_0\weireducible f \text{ and } g_0\weireducible g \}. \]
This operator captures the idea of using $g$ and $f$ in series, possibly using a computable procedure to map a name for an output of $g$ to a name for an input of $f$. While, formally, $f\compproduct g$ is a Weihrauch degree (and not a specific multi-valued function), with a small abuse of notation, we write $h\weireducible f\compproduct g$ with the obvious meaning. We also write $f^{[n]}$ to denote the $n$-fold compositional product of $f$ with itself, where $f^{[0]}:=\mathrm{id}$ and $f^{[1]}:=f$.

	Let $\boldfaceGamma$ be a Borel pointclass. We say that $f\pfunction{X}{Y}$ is $\boldfaceGamma$-measurable if, for every open $U\subset Y$, $f^{-1}(U)\in \boldfaceGamma(\dom(f))$, i.e.\ there exists $V\in\boldfaceGamma(X)$ s.t.\ $f^{-1}(U)=V\cap \dom(f)$. If $X$ and $Y$ are represented spaces, we say that $f$ is \textdef{effectively $\boldfaceGamma$-measurable} or \textdef{$\boldfaceGamma$-computable} if the map
	\[ \boldfaceGamma^{-1}(f)\mfunction{\boldfaceSigma^0_1(Y)}{\boldfaceGamma(X)}:=U\mapsto\{V\subset X \st f^{-1}(U)=V\cap \dom(f) \} \]
	is computable. In particular, if $f$ is total then $\boldfaceGamma^{-1}(f)$ is single-valued. This notion can be generalized in a straightforward way to multi-valued functions (see \cite[Def.\ 3.5]{Brattka05}).
	
	Let $\mflim\pfunction{\infStrings{(\Baire)}}{\Baire}$ be the function mapping a convergent sequence in the Baire space to its limit. In the proof of \thref{thm:wei_degree_hausdorff_fourier} we will use the following result:
	
	\begin{theorem}[{\cite[Thm.\ 6.5]{BGP17}}]
		\thlabel{thm:eff_measurability}
		$f \text{ is }\boldfaceSigma^0_{k+1}\text{-computable}$ iff $f$ is Weihrauch reducible to $\mflim^{[k]}$.
	\end{theorem}
	This is a generalization of \cite[Thm.\ 9.1]{Brattka05}, and draws an important connection between the Weihrauch degrees and the effective Borel hierarchy.

We identify two maps corresponding to the Hausdorff dimension and two maps corresponding to the Fourier dimension, according to the way closed sets are represented:
\begin{itemize}
	\item[] $\hdim,\fourierdim \function{\hypClosedUF(\mathbb{R}^d)}{\mathbb{R}}$,
	\item[] $\hdim^{\hypClosedF}, \fourierdim^{\hypClosedF} \function{\hypClosedF(\mathbb{R}^d)}{\mathbb{R}}$.
\end{itemize}

\begin{theorem}
	\thlabel{thm:wei_degree_hausdorff_fourier}
	$\mflim^{[2]}\weiequiv \hdim^{\hypClosedF} \weiequiv \hdim \weiequiv \fourierdim^{\hypClosedF} \weiequiv \fourierdim.$
\end{theorem}
\begin{proof}
	It is immediate to see that $\hdim^{\hypClosedF} \weireducible \hdim$ and $\fourierdim^{\hypClosedF} \weireducible \fourierdim$. To prove the reductions $\hdim\weireducible \mflim^{[2]}$ and $\fourierdim\weireducible \mflim^{[2]}$, by \thref{thm:eff_measurability} it suffices to show that the maps $\hdim$ and $\fourierdim$ are $\boldfaceSigma^0_{3}$-computable. This follows by \thref{thm:eff_complexity_haus_fourier_higher_dim} as 
	\[ \hdim^{-1}((a,b)) = \{ F\in\hypClosedF(\mathbb{R}^d) \st \hdim(F)>a \land \hdim(F)<b \}. \]
	In fact, given $a,b\in [0,d]$ we can uniformly compute a $(a\oplus b)$-computable $\repmap{\boldfaceSigma^0_3}$-name for $\hdim^{-1}( (a,b) )$.
		
	Finally, to prove that $\mflim^{[2]}\weireducible \hdim^{\hypClosedF}$ and $\mflim^{[2]}\weireducible \fourierdim^{\hypClosedF}$ we show that, given a sequence $\sequence{x_i}{i\in\mathbb{N}}$ in $\Cantor$, we can uniformly build a closed Salem subset $A$ of $[0,1]^d$ s.t.\ $\dim(A)$ uniformly computes whether $x_i \in Q_2$, where $Q_2$ is the fixed $\lightfaceSigma^0_2$-complete set (see Section~\ref{sec:rep_hyperspaces}). This suffices because $\mflim^{[2]}$ is Weihrauch equivalent to answering countably many $\lightfaceSigma^0_2$ questions in parallel.
	
	Let $f$ be the computable map provided by \thref{thm:sigma02_map_salem_eff}. Let also
	\[ r\function{\hypCompactV([0,1])}{\hypCompactV([-1,1]^d)} :=F\mapsto \{ z\in \mathbb{R}^d \st |z| \in F \lor |z|=1\}\] and define $g:= r\circ f$.  Recall that, by the classic theorem of Gatesoupe \cite{Gatesoupe67} (to apply it we added the condition $|z|=1$ in the definition of $r$), if $F$ is Salem with dimension $\alpha$ then $r(F)$ is Salem with dimension $d-1+\alpha$. For every non-constantly $0$ string $\sigma\in \cantor$, let $I_\sigma:=\{ i<\length{\sigma} \st \sigma(i)=1\}$ and let $p_\sigma:=\frac{1}{15}+\sum_{i\in I_\sigma} 2^{-2i-1}$. The term $1/15$ is added for technical reasons, which will become apparent at the end of the proof. Define $y_\sigma\in \Cantor$ as $y_\sigma(n):=\max_{i\in I_\sigma} x_i(n)$. Clearly
	\begin{align*}
		(\forall i\in I_\sigma)(x_i\in Q_2) & \iff y_\sigma \in Q_2\\
		& \iff \dim(f(p_\sigma,y_\sigma))=p_\sigma \iff\dim(g(p_\sigma,y_\sigma))=d-1+p_\sigma~.
	\end{align*}
	On the other hand, $(\exists i\in I_\sigma)(x_i\notin Q_2) \iff \dim(g(p_\sigma,y_\sigma))=d-1$.
	
	As in the proof of \thref{thm:eff_completeness_n_dim}, let $\sequence{C_n}{n\in\mathbb{N}}$ be a sequence of mutually disjoint closed cubes s.t.\
	\begin{itemize}
		\item $C_n\subset [0,1]^d$,
		\item $\closure{\bigcup_{n\in\mathbb{N}} C_n} = \{\mathbf{0} \} \cup \bigcup_{n\in\mathbb{N}} C_n$, where $\mathbf{0}$ is the origin of the $d$-dimensional Euclidean space,
		\item the sets have uniformly computable $\psi$-names.
	\end{itemize}
	
	For every $\sigma$ as above, we can uniformly translate and scale the set $g(p_\sigma,y_\sigma)$ to a subset $G_\sigma$ of $C_{\coding{\sigma}}$. Consider now the closed set $A:=\{ \mathbf{0} \} \cup \bigcup_\sigma G_\sigma$. It is routine to show that $A$ is Salem and $\dim(A)=d-1+\frac{1}{15}+\sum_{i\in\mathbb{N}} 2^{-2i-1}\charfun{Q_2}(x_i)$. It is then straightforward to notice that, for every $i$, the value of $\charfun{Q_2}(x_i)$ is the $(2i)$-th digit in the binary expansion of $\dim(A)$.
	
	Notice that, in general, the map sending a Cauchy representation of a real to its binary expansion is not computable (reals with two binary representations can be used to diagonalize against any possible computation). To ensure computability, we defined $p_\sigma$ so that $\dim(A)$ is guaranteed to have a unique binary expansion (and hence its binary representation is computable from its Cauchy name). In fact the binary expansion of $\dim(A)$ has value $\charfun{Q_2}(x_i)$ in the $(2i)$-th position, $0$ in the positions congruent to $1 \mod 4$, and $1$ in the positions congruent to $3 \mod 4$ (to attain the latter conditions we added $1/15$, which in binary is $0.\overline{0001}$). 
\end{proof}

The Weihrauch equivalence between $\mflim^{[2]}$ and the map computing the Hausdorff dimension of a closed subset of $[0,1]$ (and, more generally, of a compact subset of $\mathbb{R}$) was already proved in \cite[Thm.\ 48]{PFmeasures}. Our approach extends that result and, at the same time, characterizes the degree of the map computing the Fourier dimension.

\bibliographystyle{mbibstyle}
\bibliography{bibliography}

\printauthor

\end{document}